\documentclass{article}

\usepackage{amsfonts,amsthm,amssymb}
\usepackage{amsmath}
\usepackage{fullpage}
\usepackage{enumitem}
\usepackage{graphicx}

\def\ds{\displaystyle}
\def\G{\mathcal{G}}

\def\aut{\mathrm{aut}}
\newcommand{\ex}{\mathbf{E}}
\newcommand{\var}{\mathbf{Var}}

\usepackage{tikz}
    	\usetikzlibrary{arrows}
	\usetikzlibrary{shapes,snakes}
	\usetikzlibrary{calc}
	\usetikzlibrary{decorations.markings}
	\usetikzlibrary{patterns}
	\usetikzlibrary{trees,positioning,fit}

\newtheorem{theorem}{Theorem}

\newtheorem{lemma}[theorem]{Lemma}

\title{Further results on random cubic planar graphs}

\author{
    Marc Noy \thanks{Universitat Polit\`ecnica de Catalunya and Barcelona Graduate School of Mathematics, Department of Mathematics, Edifici Omega, 08034 Barcelona, Spain. E-mail: {\tt marc.noy@upc.edu}. Supported by the Spanish Ministerio de Econom\'{i}a y Competitividad projects MTM2014-54745-P, MTM2017-82166-P  and MDM-2014-0445.}
\and
    Cl\'ement Requil\'e \thanks{Institute for Algebra, Johannes Kepler Universit\"at Linz, Austria. E-mail: {\tt clement.requile@jku.at}. Supported by the Austrian Science Fund (FWF) grant F5004.
    The work presented in this paper was carried out while the author was affiliated with the Institut f\"ur Mathematik und Informatik, Freie Universit\"at Berlin, and Berlin Mathematical School, Germany, and partially supported by the Marie Curie Career Integration Grant FP7 PEOPLE - 2013-CIG 630749 - Countgraph.}
\and
Juanjo Ru\'e \thanks{Universitat Polit\`ecnica de Catalunya and Barcelona Graduate School of Mathematics, Department of Mathematics, Edifici Omega, 08034 Barcelona, Spain. E-mail: {\tt juan.jose.rue@upc.edu}. Supported by the Spanish Ministerio de Econom\'{i}a y Competitividad project MTM2014-54745-P, MTM2017-82166-P and the Marie Curie Career Integration Grant FP7 PEOPLE - 2013-CIG 630749 - Countgraph.}
}
\date{}

\begin{document}

\maketitle

\begin{abstract}
We provide precise asymptotic estimates for the number of several classes of  labeled cubic planar graphs, and we analyze properties of such random graphs under the uniform distribution.
This model was first analyzed by Bodirsky et al. (Random Structures Algorithms 2007).  We revisit their work and obtain new results on the enumeration of cubic planar graphs  and on random cubic planar graphs.
In particular, we determine the exact probability of a random cubic planar graph being connected, and
we show that the distribution of the number of triangles in random cubic planar graphs is asymptotically normal with linear expectation and variance. To the best of our knowledge,  this is the first time one is able to determine  the asymptotic distribution for the number of copies of a fixed graph containing a cycle in classes of random planar graphs arising from planar maps.
\end{abstract}


\section{Introduction and summary of results}

The enumeration of labeled planar graphs  has been recently  the subject of much research; see \cite{ICM, handbook} for  surveys on the area.
The problem of counting planar graphs was first solved by Gim\'enez and Noy \cite{gn2009}, while cubic planar graphs where enumerated by Bodirsky, Kang, L\"offler and McDiarmid \cite{cubic}. More recently, the present authors solved the problem of enumerating 4-regular planar graphs \cite{4-regular}. Several open problems remain, like the enumeration of bipartite or triangle-free planar graphs.

The goal of this paper is to sharpen the results from \cite{cubic}, as well as to prove new results.
We first enumerate asymptotically several classes of labeled cubic planar graphs. Among our new results are the enumeration
of cubic planar multigraphs and of triangle-free cubic planar graphs. In order to achieve this goal we need to use the so-called Dissymmetry Theorem for counting unrooted graphs whose structure can be encoded by means of a decomposition tree.

Random cubic planar graphs are analyzed according to the uniform distribution. More precisely, let $\mathcal{G}$ be the class of labeled cubic planar graphs and let $g_n$ be the number of graphs in $\mathcal{G}$ with $n$ vertices. Then each graph in $\mathcal{G}$ with $n$ vertices is taken with the same probability $1/g_n$.
We obtain the exact probability that a random cubic planar graph is connected, and we  prove several results on the distribution of the number of copies of a fixed subgraph.
In particular, we show that the distribution of the number of triangles is asymptotically normal with linear expectation and variance. To the best of our knowledge,  this is the first time one is able to determine  the asymptotic distribution of the number of copies of a fixed graph $H$ containing a cycle in classes of random planar graphs arising from planar maps. We also obtain Gaussian limit laws for the number of copies of certain almost cubic subgraphs.

The proofs are based on combinatorial decompositions, generating functions and asymptotic analysis of their coefficients, using the tools of analytic combinatorics \cite{FS}. In several places we use  \textsc{Maple} to perform symbolic and numerical computations.

\subsection{Results on enumeration}

In the first place we obtain an asymptotic estimate for the number $c_n$ of connected cubic planar graphs.
In all the statements that follow, $n$ should be even since a cubic graph has necessarily an even number of vertices. To avoid repetition, we assume this is always the case when referring to the number of vertices in cubic graphs.
All the numerical constants in this paper are given with a precision of 6 decimals places.

\begin{theorem}\label{thm:count-connected}
The number $c_n$ of connected cubic planar graphs with $n$ vertices is asymptotically
$$
	c_n \sim c \cdot n^{-7/2} \gamma^n n!,
$$
where $c \approx 0.060973$ and $ \gamma =\rho^{-1} \approx 3.132591$, where $\rho \approx 0.319225$ is the smallest positive root of the equation
\begin{equation}\label{eq:rho}
	729x^{12}+17496x^{10}+148716x^8+513216x^6-7293760x^4+279936x^2+46656=0.
\end{equation}
\end{theorem}

Next we estimate the number of all cubic planar graphs.

\begin{theorem}\label{thm:count-all}
The number $g_n$ of cubic planar graphs with $n$ vertices is asymptotically
$$
	g_n \sim g \cdot n^{-7/2} \gamma^n n!,
$$
where $\gamma$ is as in Theorem \ref{thm:count-connected} and $g \approx 0.061010$.
As a consequence, the limiting probability $p$ that a random cubic planar graph is connected is equal to
$$
	p = \frac{c}{g} \approx 0.999397.
$$
\end{theorem}
We remark that the actual value of $p$ was not computed in \cite{cubic}, only estimated from values of $c_n$ and $g_n$ for small $n$. As we will see later,  $p$ can be computed exactly  using the Dissymmetry Theorem.
Once we have the value of $p$, a standard proof (see \cite{3-conn})   shows that the number of connected components in a random cubic graph is asymptotically distributed as $X+1$, where $X$ is a Poisson law of parameter $\lambda \approx 0.000604$.

It is also possible to estimate the number of 2-connected cubic planar graphs.

\begin{theorem}\label{thm:count-2connected}
Let $b_n$ be the number of 2-connected cubic planar graphs on $n$ vertices.
Then
$$
	b_n \sim b \cdot n^{-7/2} \gamma_b^n n!,
$$
where $b  \approx 0.059244$, $ \gamma_b =\rho_b^{-1} \approx 3.129666$, where $\rho_b \approx 0.319523$ is the smallest positive solution of
$$
54x^6+324x^4-4265x^2+432=0.
$$
\end{theorem}

Our next result is an estimate on the number of cubic planar \emph{multigraphs}. This class of graphs is instrumental in the study of the phase transition of the Erd\H os-R\'enyi random graph
\cite{giant,kl2012,critical}. In these references cubic multigraphs are equipped with a weight that depends on the number of loops and multiple edges. Here we count unweighted cubic multigraphs, which is a result interesting by itself.

\begin{theorem}\label{thm:count-multi}
The number $h_n$ of  cubic planar multigraphs is asymptotically
$$
	h_n \sim h \cdot n^{-7/2} \gamma_m^n n!,
$$
with $h\approx 0.224743$ and $\gamma_m =\rho_m^{-1} \approx 3.985537$, where $\rho_m \approx 0.250907$ is the smallest positive root of the equation
\begin{equation}\label{eq:rho-multi}
	729\ x^{12} - 17496\ x^{10} + 148716\ x^8 - 513216\ x^6 - 7293760\ x^4 - 279936\ x^2 + 46656 = 0.
\end{equation}
The same  estimate  holds for the number  of  connected cubic planar multigraphs, but with $h$ replaced by the constant $h' \approx 0.209410$.  The limiting probability of connectivity is
$$
	p_m = \frac{h'}{h} \approx 0.931778. 
$$
\end{theorem}

\noindent We remark that the proof needs again an application of the  Dissymmetry Theorem, since the presence of loops and multiple edges does not allow us, as for simple graphs, to directly relate the number of graphs rooted at a vertex with those rooted at an edge. In addition, the similarity between equations \eqref{eq:rho} and \eqref{eq:rho-multi} will be explained later. 

We recall that a sequence  $(a_n)$ is $P$-\emph{recursive} if it satisfies a linear recurrence relation whose coefficients are polynomials in $n$.

\begin{theorem}\label{th:Dfinite}
The following sequences are $P$-recursive: the numbers of arbitrary, connected and 2-connected cubic planar graphs, and the number of cubic planar multigraphs.
\end{theorem}
The proofs rely on the algebraic character of several of the generating functions involved and, in the case of cubic multigraphs, on a further  application of the Dissymmetry Theorem.

Our last result in this section is the enumeration of triangle-free cubic planar graphs. The proof is more involved and will be given after the proof of Theorem \ref{thm:triangles}, since it uses the techniques introduced there for studying the distribution of the number of  triangles in random cubic planar graphs.

\begin{theorem}\label{thm:count-triangle-free}
The number $u_n$  of connected triangle-free  cubic planar graphs with $n$ vertices is asymptotically
$$
	f_n \sim f \cdot n^{-7/2} \gamma_t^n n!,
$$
with $f \approx 0.000911$ and $\gamma_t = \rho_t^{-1} \approx 2.641747$,  where $\rho_t \approx 0.378537$ is the smallest positive solution of the equation
\begin{equation}\label{eq:sing-triangle-free}
\def\arraystretch{1.5}
\begin{array}{ll}
    {x}^{40} - 2{x}^{38} - 41{x}^{36} + 180{x}^{34} + 285{x}^{32} -3630{x}^{30} - {\frac {26651}{4}}{x}^{28} + {\frac {5654783}{32}}{x}^{26} \\
    - {\frac {3989098451}{4096}}{x}^{24} + {\frac {50409552353}{16384}}{x}^{22} - {\frac {246713078305261}{37748736}}{x}^{20} + {\frac {8988271236666325}{905969664}}{x}^{18} \\
    - {\frac {34616066062430108809}{3131031158784}}{x}^{16} +{\frac {148714112813428613}{16307453952}}{x}^{14} - {\frac {88102457851295}{15925248}}{x}^{12} \\
    + {\frac {28819599609215}{11943936}}{x}^{10} - {\frac {2805808889}{3888}}{x}^{8} + {\frac {130387637}{972}}{x}^{6} - {\frac {8646784}{729}}{x}^{4} - 128{x}^{2} + 64 = 0.
\end{array}
\end{equation}
In addition, the number $t_n$ of triangle-free cubic planar graphs with $n$ vertices is asymptotically
$$
	t_n \sim \alpha\cdot n^{-7/2} \gamma_t^n n!,
$$
where $\alpha \approx 0.0009109$.
\end{theorem}
The multiplicative constant $\alpha$  in the last theorem is the only constant in our work for which we do not obtain an exact expression. It would be in principle possible to obtain this expression, but the computations would be very complex. The approximate value given in the statement is estimated from small values of $n$.

At the end of the paper we provide a table with the numbers of cubic planar graphs for small values of $n$ for the new families we have enumerated: multigraphs and triangle-free graphs. The numbers for arbitrary, connected and 2-connected cubic planar graphs are listed in \cite{cubic}.

\subsection{Results on limit laws}

Given an unlabeled graph $H$, a \emph{copy} of $H$ in a labeled graph $G$ is a subgraph isomorphic to $H$.
Our  results in this section  deal with the number of copies of a fixed subgraph.
We start with the number of triangles, the main result in this section. We say that a sequence $X_n$ of random variables is asymptotically normal if the standardized  variables $(X_n - \ex[X_n])/\sigma(X_n)$ converge in distribution to the standard normal law.

\begin{theorem}\label{thm:triangles}
Let $X_n$ be the number of triangles in a random
cubic planar graph. Then $X_n$ is asymptotically normal with moments
$$
	\mathbf{E} [X_n] \sim \mu n , \qquad \mathbf{Var} [X_n] \sim \lambda n,
$$
where
$$
	\mu \approx 0.121974, \qquad \lambda \approx 0.064985.
$$
\end{theorem}
It was proved in~\cite{cubic} that $X_n$ is linear with high probability. Our result is a considerable sharpening of this fact.
The proof, based on the so-called Quasi-powers Theorem, is technically involved and we are not able to extend it, for instance, to the number of cycles of length 4. The key property here is that two triangles in a cubic graph are either vertex disjoint or share one edge.

Our final results concern the number of copies of graphs which are close to being cubic.
We define a \emph{cherry} as a planar graph in which all vertices have degree 3 except for one  vertex of degree 1. The smallest cherry has 6 vertices and is obtained by subdividing one edge of $K_4$ and attaching  one vertex of degree 1.
In what follows, we denote by $\aut(H)$ the number of automorphisms of a graph $H$.
We recall that the number of different ways of labeling an unlabeled graph $H$ is equal $n!/\aut(H)$.

\begin{theorem}\label{thm:cherries}
Let $X_{H,n}$ be the number of copies of a fixed unlabeled cherry $H$ with $h$ vertices in a random cubic planar graph.
 Then $X_{H,n}$ is asymptotically normal with moments
$$
	\mathbf{E} [X_{H,n}] \sim \mu n , \qquad \mathbf{Var} [X_{H,n}] \sim \lambda n,
$$
where
$$	
	\mu = \ds\frac{4374(\rho^4+8\rho+4)^2}{\rho^2P_1}\cdot \ds\frac{\rho^h}{\aut(H)},
	\qquad \lambda = \ds\frac{8748(\rho^4+8\rho+4)(P_2h+ P_3)}{\rho^4P_1^3 }\cdot \ds\frac{\rho^{2h}}{\aut(H)^2} + \mu + \mu^2,
$$
where $\rho$ is as in Theorem \ref{thm:count-connected}, and
$$
	\def\arraystretch{1.5}
	\begin{array}{rcl}
	P_1 &=& -(2187{\rho}^{10} +43740{\rho}^{8} + 297432{\rho}^{6} + 769824{\rho}^{4} - 7293760{\rho}^{2} + 139968) > 0,	\\
	P_2 &=& -4374\left( {\rho}^{4}+8\,{\rho}^{2}+4 \right) ^{3} P_1,\\
	P_3 &=& -14348907{\rho}^{22} - 593088156{\rho}^{20} - 10235553660{\rho}^{18} - 95276742480{\rho}^{16} - 464803389936{\rho}^{14} \\
		&& - 412656456960{\rho}^{12} + 7449015918528{\rho}^{10} + 32947458310656{\rho}^{8} - 457978474586624{\rho}^{6} \\
		&& + 18919725382656{\rho}^{4} + 3101861081088{\rho}^{2} - 19591041024.
	\end{array}
$$
Moreover, for $h\ge 2$ we have that $\lambda > 0$.
\end{theorem}

It was shown in \cite{colin2009} that, with high probability, $X_{H,n}$ is at least $cn$ for some constant $c>0$ that depends only on $H$.
Our result provides a precise limit distribution.

Define a \emph{brick} as a graph obtained from a 3-connected cubic planar graph by removing one edge, so that all vertices have degree 3 expect two vertices $u$ and $v$ that have degree 2, and such that $u$ and $v$ are distinguishable (as if the edge removed was oriented).
Our last result gives the distribution of the number of copies of a given brick.
We denote by  $K_4^-$ the graph obtained from $K_4$ by removing one edge.

\begin{theorem}\label{thm:bricks}
Let $X_{B,n}$ be the number of copies of a fixed unlabeled brick $B$, different from $K_4^-$, with $b$ vertices in a random cubic planar graph.
Then $X_{B,n}$ is asymptotically normal with moments
$$
\mathbf{E}[X_{B,n}] \sim \mu n , \qquad \mathbf{Var}[X_{B,n}] \sim \lambda n,
$$
where
$$	\mu= 	 \ds\frac{10185312\rho^2}{P_1} \ds\frac{\rho^b}{\aut(B)}, \qquad \lambda =  	 \ds\frac{242688\rho^2(P_2h+ P_3)}{P_1^3}    \ds\frac{\rho^{2b}}{\aut(B)^2} +\mu + \mu^2,
$$
 $\rho$ is as in Theorem \ref{thm:count-connected}, and
$$
\def\arraystretch{1.5}
\begin{array}{rcl}
	P_1 &=&  -(	 2187{\rho}^{10}+43740{\rho}^{8}+297432{\rho}^{6}+769824{\rho}^{4}-7293760{\rho}^{2}+139968) > 0,	\\
	P_2 &=& -854929626\rho^2P_1,\\
	P_3 &=& 880066296\,{\rho}^{20}+35202651840\,{\rho}^{18}+591404550912\,{\rho}^{16}
	+5407127322624\,{\rho}^{14}	\\
	&&+19994308272243\,{\rho}^{12}-
	51726289953708\,{\rho}^{10}-559899907432200\,{\rho}^{8}
	-	1063749220662816\,{\rho}^{6}\\
	&&-5760872476783424\,{\rho}^{4}+
	43131140739648\,{\rho}^{2}+3604751548416.
	\end{array}
$$
Moreover, for  $b\ge2$ we have that $\lambda >0$.

The same result holds for $B=K_4^-$ with constants
$$\mu \approx 0.004529, \qquad \lambda \approx 0.004343.
$$
\end{theorem}

The case when $B = K_4^-$ has to be treated separately, since it can appear in two different ways: as a 3-connected core, or as the parallel composition of two loop networks, as explained in the next section.
Bricks other than $K_4^-$ can only appear as 3-connected cores.

We have obtained similar results for parameters that have been studied for several classes of planar  and related classes of graphs \cite{3-conn}.
We can show that the number of cut vertices, the number of isthmuses (separating edges) and the number of blocks (2-connected components, including isthmuses) are all asymptotically normal with linear expectation and variance.
For the sake of brevity we omit the proofs and give only the values of the constants for the expectation and variance:

\begin{center}
	\begin{tabular}{|l|c|c|}
		\hline
		{\rm Parameter} & $\mu$ & $\lambda$ \\
		\hline
		{\normalfont Cut vertices} & {\normalfont 0.001877}  & {\normalfont 0.003793}  \\
		{\normalfont Isthmuses} &  {\normalfont 0.000939} &  {\normalfont 0.000950} \\
		{\normalfont Blocks} & {\normalfont 0.001878}  & {\normalfont 0.003796}  \\
		\hline
	\end{tabular}
\end{center}

\section{Preliminaries} \label{s:prel}

In this section we collect a number of analytic and combinatorial results that are needed in the sequel.

\paragraph{Analytic combinatorics.}

We use the elements of analytic combinatorics as in \cite{FS}.
To a class $\G$ of labeled graphs, we associate the exponential generating function $G(x) = \sum_{n\ge0} g_n x^n/n!$, where $g_n$ is the number of graphs in $\G$ with $n$ vertices.
We define $\G^\bullet$ as the class of graphs in $\G$ with a distinguished vertex (that we call the root).
By the basic rules of the symbolic method, its generating function is $G^\bullet(x) = xG'(x)$.

Given a complex number $\zeta \ne0$, a $\Delta$-domain at $\zeta$ is an open set in the complex plane of the form
$$
	\Delta(R,\phi) = \{  z \colon |z|<R, z \ne \zeta, |\arg(z-\zeta)|> \phi \}.
$$
A dominant singularity of a complex function is a singularity of the smallest modulus.
The basic tool for extracting asymptotic estimates from generating functions is the following (see \cite[Corollary VI.1]{FS}).

\begin{lemma}[Transfer Theorem]\label{th:transfer}
	Assume that $f(z)$ has a unique dominant singularity $\rho > 0$ and is analytic in a $\Delta$-domain at $\rho$.
	If $f$ satisfies, locally around~$\rho$, the estimate
	\begin{equation*}
    	f(z) \underset{z \to \rho}{\sim} (1-z/\rho)^{-\alpha},
    \end{equation*}
	with $\alpha \not\in \{0,-1,-2,\dots \}$, then the coefficients of $f(z)$ satisfy
	\begin{equation*}
    	[z^n]f(z) \underset{n\to \infty}\sim \frac{n^{\alpha-1}}{\Gamma(\alpha)} \rho^{-n}.
	\end{equation*}
\end{lemma}

If $f$ has several dominant singularities coming from pure periodicities, then the contributions from each of them must be combined (see \cite[IV.6.1]{FS}).
In our case, the periodicities are due to the fact that cubic graphs have necessarily an even number of vertices and the corresponding generating functions are even.
We will locate the (unique) positive dominant singularity $\rho$ and will add the contributions from $\rho$ and $-\rho$.

All the singularities we will encounter are of square-root type, that is, the expansion of a function at a singularity $\rho$ is of the form
$$
	f(x) = \sum_{i\ge0} f_i X^i, \qquad X = \sqrt{1-x/\rho}.
$$
The singular expansions we encounter are of  the form
$$
	f(z) = f_0 + f_2 X^2 + \cdots + f_{2k} X^{2k} + f_{2k+1} X^{2k+1} + O(X^{2k+2}),
$$
with $k=1$ or $k=2$.
The only non-analytic term is $f_{2k+1} X^{2k+1} $, and it is from this term that asymptotic estimates are derived using the Transfer Theorem.

In order to prove asymptotic normal limit laws, we need a simplified version of the so-called Quasi-powers Theorem (see \cite[Theorem IX.8]{FS}).

\begin{lemma}[Quasi-powers Theorem] \label{th:quasi-powers}
	Let $\{X_n\}_{n\geq 1}$ be a sequence of non-negative discrete random variables with probability generating functions $p_n(u)$.
	Assume that, uniformly in a fixed complex neighborhood of $u = 1$
	$$
		p_n(u) = A(u)·B(u)^n \left(1+O\left(n^{-1}\right)\right),
	$$
	where $A(u), B(u)$ are analytic at $u = 1$ and $A(1) = B(1) = 1$.
	Assume finally that $B(u)$ satisfies the condition  $B''(1) + B'(1) - B'(1)^2 \ne 0$.

	Then the distribution of $X_n$ is, after standardization, asymptotically normal, and the mean and variance satisfy
	$$
    	\ex [X_n] \sim B'(1) n,
    	\qquad
    	\var [X_n] \sim \left({B''(1) } + {B'(1) } - {B'(1)}^2 \right) n.
	$$
\end{lemma}

In our applications we will have $B(u)=\rho(1)/\rho(u)$, where $\rho(u)$ will be the dominant singularity (as a function of $z$) of a bivariate generating function $f(z,u)$. The former expressions then become
$$
    \ex [X_n] \sim \left( -\frac{\rho'(1)}{\rho(1)} \right) n,
    \qquad \var [X_n] \sim \left( -\frac{\rho''(1)}{\rho(1)} - \frac{\rho'(1)}{\rho(1)} + \left( \frac{\rho'(1)}{\rho(1)} \right)^2 \right) n.
$$

\paragraph{Planar maps and triangulations.}

We recall that a planar map is a connected planar multigraph embedded in the plane up to homeomorphism.
A map is \emph{rooted} if one of its edges is distinguished and oriented.
In this way a rooted map has a root edge and a root vertex (the tail of the root edge).
We define the root face as the face to the right of the directed root edge.
A rooted map has no automorphism, in the sense that every vertex, edge and face is distinguishable.
From now on all maps are planar and rooted.
Since maps are not labeled, the associated generating functions are ordinary.

A map is a \emph{triangulation} if it is 3-connected and every face is a triangle (one can consider more general triangulations having loops and multiple edges but they are not needed in this paper).
The dual of a triangulation is a 3-connected cubic map, since 3-connectivity in maps is preserved under duality (a map is 3-connected if it is 3-connected as a graph and it has no multiple edges).
Let $T(z)$ be the (ordinary) generating function of 3-connected triangulations together with the map consisting of a triangle, where the variable $z$ marks the number of vertices minus two.
Then, as shown by Tutte~\cite{tutte},

\begin{equation}\label{eq:Tu}
	T(z) = U(z)\left(1- 2U(z)\right),
\end{equation}
where $U$ is an algebraic function defined by
\begin{equation}\label{eq:u}
	z = U(z)(1-U(z))^3.
\end{equation}
Equation~\eqref{eq:u} has a unique solution with positive coefficients, given by
$$
	U(z) = z + 3z^2 + 15 z^3 + 91z^4 + \cdots
$$
Then
$$
	T(z) = z +z^2 + 3z^3 + 13z^4 + \cdots
$$

As shown in \cite{tutte}, the unique singularity of $U$ (and hence of $T$) is located at $\tau = 27/256$.
In particular,
$$
	U(\tau)=1/4, \qquad T(\tau) = 1/8.
$$
The singular expansion of $U(z)$ at $\tau$ is equal to
\begin{equation}\label{eq:sing_U}
	U(z)= \frac{1}{4} - \frac{\sqrt{6}}{8}Z + \frac{1}{12}Z^2 - \frac{31\sqrt{6}}{1728}Z^3 + \frac{37}{1296}Z^4 - \frac{2093\sqrt{6}}{248832}Z^5 + O(Z^6),
\end{equation}
where $Z = \sqrt{1-z/\tau}$. From Equation \eqref{eq:Tu} we obtain the singular expansion of $T(z)$ at $\tau$
\begin{equation*} 
 	T(z) =\frac{1}{8} - \frac{3}{16}Z^2 + \frac{\sqrt{6}}{24}Z^3 - \frac{13}{192}Z^4 + \frac{35\sqrt{6}}{1728}Z^5 + O(Z^6).
\end{equation*}

We also need to consider the family of 4-connected triangulations, which  are  those not containing a separating triangle (a triangle that is not a face) and having at least 6 vertices.
The smallest 4-connected triangulation is the graph of the octahedron.
The associated  generating function $T_4(z)$, where again $z$ marks vertices minus two, is equal to (see  \cite{tutte})
\begin{equation}\label{eq:Tv}
	T_4(z) = z + V(z)(V(z) - 1)(V(z) + 1)^{-2} - z^2,
\end{equation}
where $V(z)$ is given by
\begin{equation*} 
	z = V(z)(1-V(z))^2.
\end{equation*}
The unique solution with positive coefficients is
$$
	V(z) = z + 2z^2 + 7z^3 + 30z^4 + \ldots,
$$
and
$$
	T_4(z) = z^4 + 3z^5 + 12z^6 + 52z^7 + \ldots
$$
The unique singularity of $T_4$ is at $\varsigma = 4/27$ and we have
$$
	V(\varsigma) = 1/3, \qquad T_4(\varsigma) = 7/5832.
$$
The singular expansion of $V(z)$ at $\varsigma$  is equal to
$$
V(z) =	\frac{1}{3} - \frac{2\sqrt{3}}{9}Z + \frac{2}{27}Z^2 - \frac{5\sqrt{3}}{243}Z^3 + \frac{16}{729}Z^4 - \frac{77\sqrt{3}}{8748}Z^5 + O(Z^6), \qquad Z = \sqrt{1-z/\varsigma}.
$$
As before, using  \eqref{eq:Tv} we obtain
\begin{equation*} 
 T_4(z)=   \frac{7}{5832} - \frac{245}{23328}Z^2 + \frac{\sqrt{3}}{96}Z^3 - \frac{833}{93312}Z^4 - \frac{\sqrt{3}}{864}Z^5 + O(Z^6).
\end{equation*}

\paragraph{3-connected cubic planar graphs.}

Let $M(x,y)$ be the GF of labeled 3-connected cubic planar graphs rooted at a directed edge, where $x$ marks vertices and $y$ marks edges.
There is a bijection between triangulations and planar 3-connected cubic maps given by duality.
Also, by Whitney Theorem, every 3-connected cubic planar graph admits a unique embedding in the plane up to orientation.
Using this fact we can express $M(x,y)$ in terms of the generating function $T(z)$ of rooted unlabeled triangulations, where $z$ counts the number of  vertices minus two.
The relation is
 \begin{equation}\label{eq:TM}
 	M(x,y) = \frac{1}{2}\left(T(x^2y^3) - x^2y^3\right).
 \end{equation}
 The subtracted term $x^2y^3$ corresponds to the triangulation consisting of a single triangle.
 We have
 $$
 	M(x,y) = 12\, \frac{x^4}{4!}\, y^6 + 1080\, \frac{x^6}{6!}\, y^9 + \cdots
 $$
 The first monomial corresponds to $K_4$ (a unique labeling and 12 possible roots) and the second one to the triangular prism (60 ways to label and 18 roots).

 We will also need the generating function $\overline{M}(x,y)$ of (unrooted) labeled 3-connected cubic planar graphs, which is obtained by integration.
 We have $M(x,y) = 2y \partial\overline{M}(x,y) / \partial y$, hence
 $$
 	\overline{M}(x,y) = \frac{1}{2}\int \frac{M(x,y)}{y}dy = \frac{1}{4}\int\frac{T(x^2y^3) - x^2y^3}{y} dy.
 $$
We change variables as  $z = x^2y^3$ and are left with the integral $\frac{1}{12} \int T(z)/z \, dz$.
We make the further change $v = U(z)$ and, using Equations \eqref{eq:Tu} and \eqref{eq:u}, we get
 \begin{align*}
 	\overline{M}(x,y)
 	& = \frac{1}{12} \left(\int\frac{T(z)}{z} dz - z\right) \\
 	& = \frac{1}{12}\left(\int\frac{(1-2v)(1-4v)}{1-v} dv - z \right)\\
 	& = -\frac{1}{12}\left( 4v^2 + 2v + 3\log(1-v) + z\right).
 \end{align*}
Hence
 \begin{equation}\label{eq:MM}
     \overline{M}(x,y) = -\frac{1}{12}\left( 4U(x^2y^3)^2 + 2U(x^2y^3) + 3\log(1-U(x^2y^3)) +x^2y^3 \right).
 \end{equation}

\paragraph{Networks.}
We follow the definitions from \cite{cubic} but deviate slightly from the notation there.
A \emph{network} is a connected cubic planar multigraph $G$ with an ordered pair of adjacent vertices $(s,t)$  such that the graph obtained by removing the edge $st$ is simple. There could be an additional edge between $s$ and $t$ which is not removed. 
We notice that $st$ can be a simple edge, a loop or a belong to a double edge.
The oriented edge $st$ is the \emph{root} of the network  and $s,t$ are the \emph{poles}.

Given a network $H$, with root edge $st$, and a directed edge $e=uv$ of another network $G$, the \emph{replacement} of $e$ with $H$ is the network obtained from $G$ by performing the following operation.
Subdivide the edge $uv$ twice producing a path $uu'v'v$, remove the edge $u'v'$, and identify $u'$ and $v'$, respectively, with vertices $s$ and $t$ of $H-st$.
Notice that if $G$ and $H$ are cubic and planar, so is the resulting network.

A cut vertex in a cubic graph is necessarily incident with one or three isthmuses.
For each cut vertex $u$ incident with exactly one isthmus $e$, we can remove the component containing $e$ and erase the resulting vertex of degree 2 resulting in a cubic graph.
We call this operation \emph{suppressing} the cut vertex $u$.

By classifying the possible situations obtained by removing the edge $st$, networks fall into five classes, as shown in \cite{cubic}.
For the sake of completeness we offer an alternative proof based on  Tutte's  decomposition of 2-connected graphs into 3-connected components \cite{dissymetry}.

\begin{lemma}
Let $G$ be a network and let $st$ be the root edge. Then $G$ belongs to one and only one of the following classes.
\begin{itemize}
    \item $\mathcal{L}$ (Loop). The root edge is a loop.
    \item $\mathcal{I}$ (Isthmus). The root edge is an isthmus.
    \item $\mathcal{S}$ (Series). $G-st$ is connected but is not 2-connected.
    \item $\mathcal{P}$ (Parallel). $G-st$ is 2-connected and $G - \{s,t\}$ is not connected.
    \item $\mathcal{H}$ (3-connected). $G$ is obtained from a 3-connected graph by possibly replacing each non-root edge with a network of types $\mathcal{L}$, $\mathcal{S}$, $\mathcal{P}$ or $\mathcal{H}$.
\end{itemize}

\end{lemma}

\begin{proof}
Let $G$ be a network with root edge $st$, and suppose $st$ is neither a loop nor an isthmus, so we are not in the classes $\mathcal{L}$ or $\mathcal{I}$.
Consider the 2-connected core $C$ obtained by suppressing all cut vertices incident with exactly one isthmus.
By Tutte's decomposition into 3-connected components, $C$ belongs to either $\mathcal{S}, \mathcal{P}$ or $\mathcal{H}$.
\end{proof}
Let now $\mathcal{D}$ be the class of networks for which the graph resulting from the removal of the root edge remains connected.
It is by definition,
$$
	\mathcal{D} = \mathcal{L} + \mathcal{S} +
	\mathcal{P} + \mathcal{H},
$$
where $+$ denotes the disjoint union of classes, and the class $\mathcal{I}$ is excluded since removing the root edge of networks in this class disconnects the graph.
Let then $L(x)$, $I(x)$, $S(x)$, $P(x)$, $H(x)$, $D(x)$  be the associated generating functions.

The following result, based on simple combinatorial arguments, is shown in \cite[Section 3]{cubic}.

 \begin{lemma}\label{lem:networks}
 The following equations hold:
\begin{equation}\label{eq:networks}
\def\arraystretch{2}
\begin{array}{lll}
	D &=& L+S+P+H, \\
	L &=& \ds\frac{x^2}{2} (I+D-L), \\
	S &=& D(D-S), \\
	I &=& \ds\frac{L^2}{x^2}, \\
	P &=& x^2 D + \ds\frac{x^2}{2} D^2, \\
	H &=& \ds\frac{M(x,1+D)}{1+D}.
\end{array}
\end{equation}
 \end{lemma}

Notice that all the functions involved are even, in agreement with the fact that a cubic graph has an even number of vertices.
Using the relations $D-L=S+P+H$ and $D-S = L+P+H$, the system \eqref{eq:networks} can be rewritten so that all the functions on the right hand-side have non-negative coefficients when expanded in terms of $x,L,I, S, H$ and $D$.
This is also true for the equation $H = M(x,1+D)/(1+D)$, since $M(x,y)$ is divisible by $y$.
It follows (see \cite{drmota}) that there is a unique solution of the system with non-negative coefficients, which is the combinatorial solution.


Let $C(x)$ be the generating function of connected cubic planar graphs, and $C^\bullet(x)=x C'(x)$ that of connected graphs rooted at a vertex.
As shown in \cite{cubic}, $C^\bullet(x)$ can be expressed in terms of networks as
\begin{equation}\label{eq:C rooted}
	3C^\bullet(x) = D(x) + I(x)-L(x)  -x^2 D(x) - L(x)^2.
\end{equation}
The factor 3 comes from double counting since at every root vertex $v$ we have 3 possible root edges with $v$ as a tail.
The term $D(x)+I(x)$ encodes all types of networks, from which one has to subtract those which are not simple.
These are $\mathcal{L}$, where the root edge is a loop, and those where the root edge is a double edge: parallel networks encoded by $x^2D(x)$, and series networks encoded by $L(x)^2$.

\paragraph{The Dissymmetry Theorem for tree-decomposable classes.}

We follow the formulation in \cite{dissymetry}.
A class of graphs ${\cal A}$ is said to be {\em tree-decomposable} if for each graph $\gamma\in {\cal A}$, we can associate in a unique way a tree $\tau(\gamma)$ whose nodes are distinguishable (for instance, by using the labels on the vertices of $\gamma$).
Let ${\cal A}_{\bullet}$ denotes the class of graphs in ${\cal A}$ where a node of $\tau(\gamma)$ is distinguished.
Similarly, ${\cal A}_{\bullet - \bullet}$ is the class of graphs in ${\cal A}$ where an edge of $\tau(\gamma)$ is distinguished, and ${\cal A}_{\bullet \rightarrow \bullet}$ those where an edge $\tau(\gamma)$ is distinguished and given a direction.
The {\em Dissymmetry Theorem for trees} by \cite{species} allows us to express the class of unrooted trees in terms of the classes of trees with a distinguished vertex, edge and directed edge.
This result can be extended to tree-decomposable classes in the following way (see \cite{dissymetry}).

\begin{theorem}\label{thm:species}
Let ${\cal A}$ be a tree-decomposable class.
Then
$$
    {\cal A} + {\cal A}_{\bullet \rightarrow \bullet} \simeq {\cal A}_{\bullet} + {\cal A}_{\bullet - \bullet},
$$
where $\simeq$ is a bijection preserving the number of nodes.
\end{theorem}

\section{Proofs of enumerative results}

Before proving Theorem \ref{thm:count-connected}, we mention the corresponding proof presented  in  \cite{cubic}.
It is shown  by algebraic elimination from the system \eqref{eq:networks}, that one can obtain a single polynomial equation $\Phi(x,D)$ satisfied by $D$.
It is then claimed in \cite{cubic} that the root $\rho \approx 0.319224$ of the discriminant of $\Phi$ with respect to $D$ provides the dominant singularity of $D$.
However, the analysis is incomplete since one must guarantee that this is the only dominant singularity and that there are no smaller singularities arising from a potential branch point.
Then a singular expansion of $D(x)$ at $\rho$ of square-root type is somehow guessed  as
$$
	D(x) = D_0 + D_2 X^2 + D_3 X^3 + O(X^4), \qquad
 	X = \sqrt{1-x/\rho}.
$$
The validity of this expansion is not fully established and the coefficients $D_i$ are apparently not computed in~\cite{cubic}.
Using the Transfer Theorem, an estimate for the coefficients of $D$ is derived, which implies a corresponding estimate for the coefficients of $C$.
This remark also applies to the proof of Theorem \ref{thm:count-2connected}.

In order to provide completely rigorous proofs for this and the subsequent enumeration results, we will use the following technical lemma.
We recall that the discriminant of an algebraic function $A$ is characterized by the common roots of the  minimal polynomial of $A$ and its derivative with respect to $A$.
Recall also that $\tau=\frac{27}{256}$ is the singularity of the generating function $T(z)$ of triangulations.

\begin{lemma}\label{lem:technical}
Let $A(x)$ be an even algebraic power series with positive coefficients which satisfies an  equation of the form
	$$
    H(x,A(x),T(x^2(1+A(x))^3))=0,
	$$
where $H$ is a polynomial, and $T$ is the generating function of triangulations as in (\ref{eq:Tu}). Let $p(x)$ be the discriminant of~$A$.
	
Assume that the equations $H=0$ and $x^2(1+A(x))^3 = \tau$ have a common positive solution, and let $(x_0,A_0)$ be the one with smallest $x_0$ value.
Assume in addition the following conditions:
\begin{enumerate}
\item $x_0$ is the smallest positive root of $p(x)$, and  $\pm x_0$ are the only  roots of $p(x)$ of modulus $x_0$.
\item $H_A(x_0,A(x_0), T(x_0^2(1+A(x_0))^3))\ne0$, where $H_A$ is the derivative of $H$ with respect to the second variable.
\item $A'(x_0) $ is finite and $A''(x_0)=+\infty$, where both evaluations are taken as limits as $x \to x_0^-$.

\end{enumerate}
\noindent
Then $x_0$ is the unique dominant singularity of $A(x)$, and the expansion at $x_0$ is of the form
	$$
		A(x) = A_ 0 + A_2X^2 + A_3 X^3 + O(X^4),
	$$
	where $X= \sqrt{1-x/x_0}$, and $A_3 >0$ is a computable  algebraic number.
	Furthermore, the following asymptotic estimate holds for $n$ even:
	$$
		[x^n] A(x) \sim \frac{3A_3}{2 \sqrt{\pi}}  n^{-5/2} x_0^{-n}.
	$$
\end{lemma}


\begin{proof}
The potential singularities of an algebraic functions are among the roots of its discriminant \cite[Section VII.7.1]{FS}.
Since $x_0$ is the smallest positive root of $p(x)$, $A(x)$ is analytic in the disk $|x| < x_0$.
Since $A''(x_0)=+\infty$ (recall that by hypothesis $A(x)$ has positive coefficients), $A(x)$ is not analytic at $x_0$. It follows that $x_0$ is a dominant singularity.
The condition $H_A(x_0,A(x_0),T(x_0^2(1+A(x_0))^3))\ne0$ guarantees that $x_0$ is not at the same time a branch point when solving $H=0$.
Since $T(z)$ has an expansion at $\tau$ in powers of $\sqrt{1-z/\tau}$, by indeterminate coefficients on the singular exponents $A(x)$ has a Puiseux expansion at $x=x_0$ of the form $\sum_{i\ge0} A_iX^i$ at $x_0$.
The condition $A'(x_0)< + \infty$ implies that  $A_1=0$, and $A''(x_0) =+ \infty$ implies that $A_3 \ne 0$, as claimed.
The coefficients $A_i$ are algebraic numbers since $A$ is an algebraic function.
	
Because of the expansion in powers of $X$, $A(x)$ is analytic in a neighborhood of $x_0$ slicing the ray $[x_0,+\infty]$.
Since $p(z)$ has no other root of modulus $x_0$, it follows that $\pm x_0$ are the only singularities satisfying that $|x|=x_0$.
A standard compactness argument (see the last part of the proof of Theorem 2.19 in \cite{drmota}) shows that $A(x)$ is analytic in a $\Delta$-domain at both $x_0$ and $-x_0$.
Hence we can apply the Transfer Theorem \cite[Corollary 6.1]{FS}, and obtain the estimate for $n$ even as claimed, using $\Gamma(-3/2) = 4 \sqrt{\pi}/3$.
Notice that the contributions from $x_0$ and $-x_0$ are added, so that the multiplicative constant is $2A_3/\Gamma(-3/2)$.
\end{proof}

\paragraph{Note.}

When computing Puiseux expansions of an algebraic function with \textsc{Maple} it may be that several solutions appear due to the different branches at a given point.
In all our proofs we find a single expansion containing a non-zero term $A_3X^3$, which has to correspond to the  branch of the combinatorial solution due to the above considerations.

\subsection{Connected cubic planar graphs}\label{subsec:simple-connected}


\paragraph{Proof of Theorem \ref{thm:count-connected}.}

We start by obtaining a single equation for $D$ from the system \eqref{eq:networks}.
First we combine the second and fourth equations and solve for $L$ as
$$
    L = 1 + \frac{x^2}{2} - \sqrt{\frac{x^4}{4} + 1 - x^2(D-1)}.
$$
The negative square-root is chosen so that $L$ has non-negative coefficients.
Then we have
$$
	D = \frac{D^2}{1+D} + x^2D + \frac{x^2}{2} D^2 + 1 + \frac{x^2}{2} - \sqrt{\frac{x^4}{4} + 1 - x^2(D-1)} + \ds\frac{M(x,1+D)}{1+D}.
$$
A simple manipulation together with \eqref{eq:TM} gives the equation
\begin{equation}\label{eq:D-connected}
	F(x,D) = (1+D)\sqrt{\frac{x^4}{4} + 1 - x^2(D-1)} - \frac{T(x^2(1+D)^3)}{2} - 1 = 0.
\end{equation}
This equation can be written as 
$$H(x,D, T(x^2(1+D)^3))=\left(1+\frac{1}{2}T\left(x^2(1+D)^3\right)\right)^2-(1+D)^2\left(\frac{x^4}{4}+1-x^2(D-1)\right)=0,$$
and $H$ is a polynomial. 
We can apply Lemma \ref{lem:technical} with $H$, however some of the computations will be done with respect to $F$.

The equations $F=0$ and $x^2(1+D(x))^3 = \tau$ have a unique positive solution, given by
$$
	\rho \approx 0.319225, \qquad D_0 = D(\rho) \approx 0.011526,
$$
We now proceed to check that the conditions of Lemma \ref{lem:technical} hold.
We eliminate from (\ref{eq:D-connected}) and (\ref{eq:Tu}) and obtain the minimal polynomial $p(x)$ of $D$, which is equal to the one displayed  in the statement.
We check (algebraically) that $\rho$ is a root of $p(x)$, and check (numerically) that $\rho$ is the root with smallest modulus, and that $\pm\rho$ are the only roots of modulus $\rho$.
We use the relations $T'(z) = (1-U(z))^{-2} = (1-U(z))U(z)/z$ and $U(\rho^2(1+D_0)^3) =1/4$ to compute
$$
	F_D(\rho,D_0) = \sqrt{\frac{\rho^4}{4} + 1 - \rho^2(D_0-1)} - \frac{\rho^2(1+D_0)}{2\sqrt{\frac{\rho^4}{4} + 1 - \rho^2(D_0-1)} } - \frac{9}{32(1+D_0)}
	\approx 0.734617 \ne 0.
$$
Next we differentiate $F$ with respect to $x$ and solve for $D'(x)$ to obtain $D'(\rho) \approx 0.370297 < +\infty $.
The derivative $F_x(x,D(x))$ contains the term $U(x^2(1+D(x)^3)$, hence the second derivative contains the term $U'(x^2(1+D(x))^3)$.
It follows that the expression for $D''(\rho)$ contains the term $U'(\rho(1+D_0)^3) = U'(\tau)$, which is infinite because of \eqref{eq:sing_U}. All the other terms, including $D'(\rho)$, remain finite, hence $D''(\rho) =+\infty$.

We compute the Puiseux expansion of $D(x)$ at $\rho$
 $$
 	D(x) = D_0 + D_2 X^2 + D_3 X^3 + O(X^4), \qquad X = \sqrt{1-x/\rho},
 $$
and obtain $D_3 \approx 0.254267$.

Finally, plugging the singular expansions of $D$ into Equation \eqref{eq:C rooted} we obtain the expansion
\begin{equation}\label{eq:expansionCrooted}
    C^\bullet(x) = C_0^\bullet + C_2^\bullet X^2 + C_3^\bullet X^3 + O(X^4),
\end{equation}
 where $C_3^\bullet \approx 0.072048$.
For $n$ even we deduce the estimate
$$
 	n\cdot c_n = n! [x^n] C^\bullet(x) \sim {3C_3^\bullet \over 2\sqrt{\pi}}\cdot n^{-5/2}\cdot \rho^{-n} n! ,
$$
and
$$
	c_n \sim c  \cdot n^{-7/2} \gamma^{n} n!, \qquad c = 3C_3^\bullet/(2\sqrt{\pi}) \approx 0.060973, \qquad \gamma = \rho^{-1} \approx 3.132591.
$$
This concludes the proof of Theorem \ref{thm:count-connected}. \qed

\subsection{Cubic planar graphs}\label{subsec:simple-general}

In order to prove Theorem \ref{thm:count-all} we need an expression of $C(x)$ in terms of the generating functions of networks.
Given Equation \eqref{eq:C rooted}, it would be sufficient to integrate $D(x)$, which is an algebraic function, but we have not been able to solve this integration problem.
This may be due to the fact that the algebraic equation defining $D(x)$ has genus 20; in a similar situation when integrating the generating function of  general planar networks, the corresponding curve has genus 0 (see \cite[page 319]{gn2009}) and determines a rational curve.
Additionally, observe that by integrating the singular expansion for $xC'(x)$ will not lead the coefficient $C_0$, which is necessary to obtain the asymptotic estimate for $[x^n]G(x)$.

Instead, we use the Dissymmetry Theorem.
This approach is more combinatorial and has the additional advantage of allowing us to prove Theorem \ref{thm:count-multi}, where the algebraic techniques used in \cite[Corollary 1.2]{4-regular} do not apply.

The key tool is to associate, to a cubic planar graph $\gamma$, a canonical tree $\tau(\gamma)$ and to encode its different rootings using the generating functions of networks introduced before.
We follow the development and terminology from \cite{dissymetry}, adapted to our situation, where the main novelty is that, due to their bounded degree, there is a finite number of cases to encode cut-vertices using networks.

\begin{figure}[htb]
    \centering
    \includegraphics[scale=0.9]{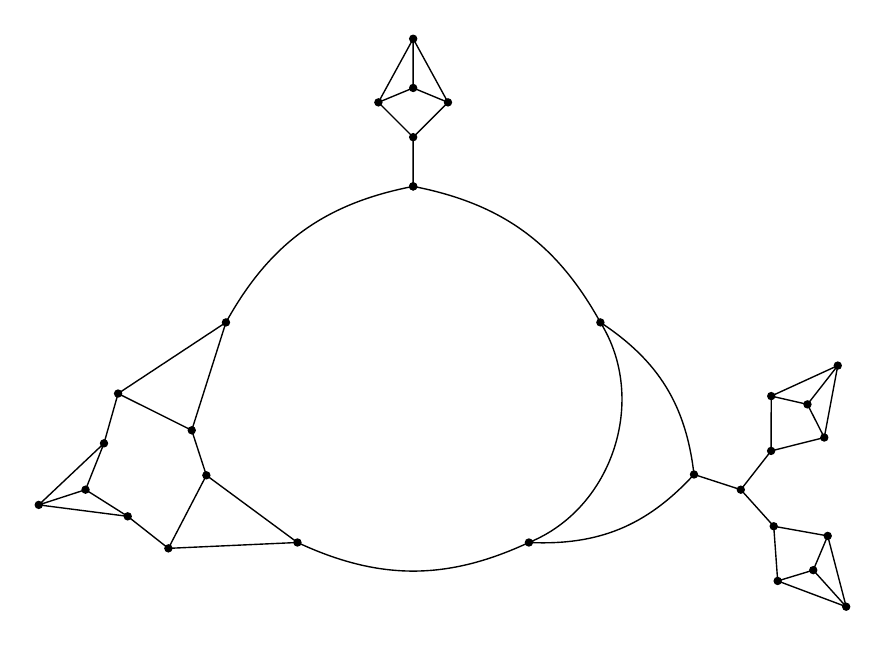}
    \includegraphics[scale=0.9]{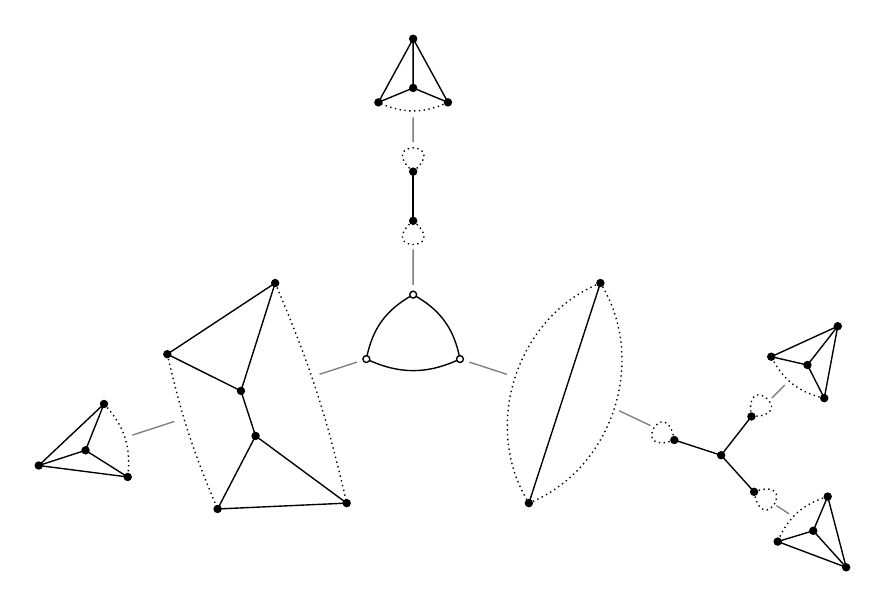}
   \includegraphics[scale=1]{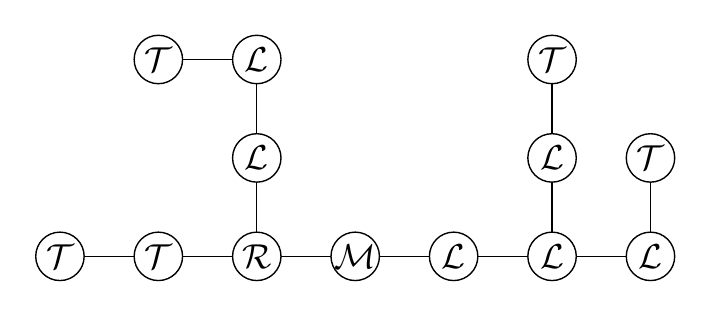}
    \caption{A connected cubic planar graph $\gamma$ and its tree-decomposition $\tau(\gamma)$.}
    \label{fig:tree-decomposition}
\end{figure}

\medskip\noindent{\it Rooting at a vertex.}
The tree $\tau(\gamma)$ has four different types of nodes, namely $\mathcal{R}$, $\mathcal{M}$, $\mathcal{T}$ and $\mathcal{L}$ corresponding respectively to the series, parallel, 3-connected and loop constructions, as illustrated by Figure \ref{fig:tree-decomposition}.
\medskip

An ${\cal R}$-node is a cycle of length at least 3 in which we replace every vertex with a network of type $\mathcal{D} - \mathcal{S}$. Notice that, by maximality of the series construction, two ${\cal R}$-nodes cannot be adjacent in the tree.
The generating function counting trees where an ${\cal R}$-node is distinguished is given by
$$
    C_{\cal R} =\text{Cycl}_{\geq 3}(D-S) =
\frac{1}{2} \left( \log\frac{1}{1-(D-S)} - (D-S)) - \frac{(D-S)^2}{2} \right).
$$

An ${\cal M}$-node is a {\em 3-bond graph} (a graph with two vertices connected by three parallel edges) in which we replace at least two of its edges with a network of type $\mathcal{D}$.
The generating function counting trees where an ${\cal M}$-node is distinguished is given by
$$
    C_{\cal M} = \frac{x^2}{2} \left( \frac{D^2}{2} + \frac{D^3}{6} \right).
$$

An ${\cal T}$-node encodes a 3-connected cubic planar graph, the {\em core}, in which every edge is (possibly) replaced by a network of type $\mathcal{D}$.
The generating function counting trees where a ${\cal T}$-node is distinguished is given by
$$
    C_{\cal T} = \overline{M}(x,1+D),
$$
where $\overline{M}$ is as in Equation \eqref{eq:MM}.

\medskip

An ${\cal L}$-node encodes a cut-vertex of $\gamma$ which separates the graph into two or three connected components.
The first case is illustrated by the leftmost graph of Figure \ref{fig:L_nodes} and is obtained by replacing the root of a loop-network with a network of type $\mathcal{D}-\mathcal{L}$ (it cannot be another loop-network as it would create a double-edge).
The second case is illustrated by the middle graph of Figure \ref{fig:L_nodes} and is obtained by gluing together three loop-networks.
The generating function counting trees where an ${\cal L}$-node is distinguished is given by
$$
    C_{\cal L} = \frac{L(D-L)}{2} + \frac{L^3}{6x^2}.
$$

\begin{figure}[htb]
\centering
\begin{minipage}{.25\textwidth}
    \centering
    \includegraphics[scale=1]{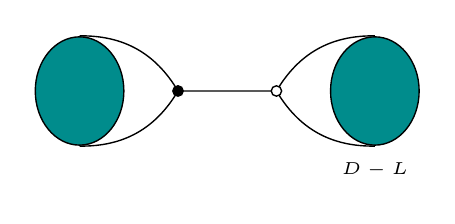}
\end{minipage}\qquad
\begin{minipage}{.25\textwidth}
    \includegraphics[scale=1]{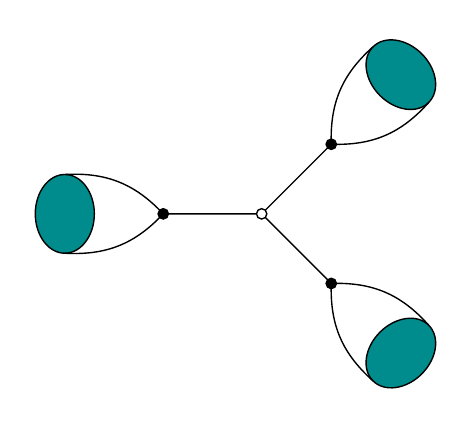}
\end{minipage}\qquad
\begin{minipage}{.25\textwidth}
    \includegraphics[scale=1]{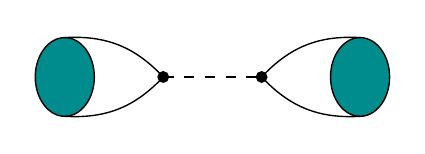}
\end{minipage}
\caption{In each of the two leftmost graphs, the white cut-vertex encodes one of the two types of ${\cal L}$-nodes. The dashed edge of the rightmost graph encodes an edge between two $\mathcal{L}$-nodes.}
\label{fig:L_nodes}
\end{figure}

\medskip\noindent{\it Rooting at an edge.}
The endpoint of an edge of $\tau(\gamma)$ can be any combination of pairs of nodes with the exception of $\mathcal{R}$--$\mathcal{R}$.

Let us describe how to get $C_{{\cal R}-{\cal M}}$. The adjacency between a ${\cal R}$-node and a $\cal M$-node can be described as an unordered pair of a series-network an a parallel network. This gives $\frac{1}{2}SP$. The same argument applies for the rest of the families where the vertices of the rooted edge are different. Finally, when the two vertices of the rooted edge are equal (case ${\cal T}-{\cal T}$ and ${\cal L}-{\cal L}$) we need to introduce an extra factor $1/2$ due to the extra symmetry of having in both sides of the edge the same combinatorial class.

 Resuming, the associated generating functions are listed below:
%
$$
\def\arraystretch{1.5}
\begin{array}{lllllllllll}
C_{{\cal R}-{\cal M}} &=& \frac{1}{2} S P, & \qquad & C_{{\cal R}-{\cal T}} &=&  \frac{1}{2} S H, &\qquad& C_{{\cal R}-{\cal L}} &=& \frac{1}{2} S L,
\\
C_{{\cal M}-{\cal T}} &=&  \frac{1}{2} P H, &\qquad &C_{{\cal M}-{\cal L}} &=& \frac{1}{2} P L, &\qquad& C_{{\cal M}-{\cal M}} &=& \frac{1}{4} P^2,
\\
C_{{\cal T}-{\cal L}} &=& \frac{1}{2} H L, &\qquad& C_{{\cal T}-{\cal T}} &=& \frac{1}{4} H^2, &\qquad& C_{{\cal L}-{\cal L}} &=& \frac{1}{2}\frac{L^2}{x^2}.
\end{array}
$$

For instance, $\mathcal{R}$--$\mathcal{M}$ corresponds to an unordered pair of a series and a parallel network, and similarly for the remaining expressions.

\medskip\noindent{\it Rooting at an oriented edge.}
If $\mathcal{A}$ and $\mathcal{B}$ are two nodes of different types, then $C_{{\cal A}\rightarrow {\cal B}} = C_{{\cal B}\rightarrow {\cal A}}$ and $C_{{\cal A}\rightarrow {\cal B}} = C_{{\cal A}-{\cal B}}$, because there are no symmetries.
When ${\cal A}={\cal B}$, we have
$C_{{\cal A}\rightarrow {\cal A}} =2 C_{{\cal A}-{\cal A}}$ because there are two possible orientations, hence
$$
    \def\arraystretch{1.5}
    \begin{array}{llll}
        C_{{\cal M}\rightarrow {\cal M}} = \frac{1}{2} P^2, &\qquad
        C_{{\cal T}\rightarrow {\cal T}} = \frac{1}{2} H^2, &
        \qquad C_{{\cal L}\rightarrow {\cal L}} = \frac{L^2}{x^2}.
    \end{array}
$$

\paragraph{Proof of Theorem \ref{thm:count-all}.}
Recall that $C$ is the generating function of unrooted connected cubic planar graphs.
A direct application of Theorem \ref{thm:species} to  the tree-decomposition described above  gives ${\cal C} + {\cal C}_{\bullet \rightarrow \bullet} \simeq {\cal C}_{\bullet} + {\cal C}_{\bullet - \bullet}$.
Translated into the associated generating functions, this yields
\begin{equation*}
\def\arraystretch{1.5}
\begin{array}{lll}
C &=& C_{\cal R} + C_{\cal M} + C_{\cal T} + C_{\cal L} + C_{{\cal M}-{\cal M}} + C_{{\cal T}-{\cal T}} + C_{{\cal L}-{\cal L}}\\
&& +C_{{\cal R}-{\cal M}} + C_{{\cal R}-{\cal T}} + C_{{\cal R}-{\cal L}} + C_{{\cal M}-{\cal T}} + C_{{\cal M}-{\cal L}} + C_{{\cal T}-{\cal L}}\\
&& -2(C_{{\cal R}\rightarrow {\cal M}} + C_{{\cal R}\rightarrow {\cal T}} + C_{{\cal R}\rightarrow {\cal L}} + C_{{\cal M}\rightarrow {\cal T}} + C_{{\cal M}\rightarrow {\cal L}} + C_{{\cal T}\rightarrow {\cal L}})\\
&& -(C_{{\cal M}\rightarrow {\cal M}} + C_{{\cal T}\rightarrow {\cal T}} + C_{{\cal L}\rightarrow {\cal L}}).
\end{array}
\end{equation*}
And using the previous expressions, this becomes
\begin{equation}\label{eq:dissymetry_simple}
\def\arraystretch{1.8}
\begin{array}{lll}
C(x) &=&
\ds\frac{x^2}{2}\left(\frac{D^2}{2} + \frac{D^3}{6}\right) + \overline{M}(x,1+D) + \frac{L^3}{6x^2}\\
&&-\ds\frac{1}{2}\left( \log(1 - D + S) + (D - S) + \frac{(D-S)^2}{2} + P(S + H) + HS + \frac{P^2 + H^2}{2} + \frac{L^2}{x^2} \right).
\end{array}
\end{equation}

Using the expansion in powers of $X=\sqrt{1-x/\rho}$ of each term on the right-hand side of Equation \eqref{eq:dissymetry_simple}, we compute the singular expansion of $C(x)$ at $\rho$, which is of the form
$$
    C_0 + C_2X^2 + C_4X^4 + C_5X^5 + O(X^6),
$$
where $C_0 \approx 0.000604$ and $C_5 \approx -0.028819$.
The fact that $C_3=0$ follows by integrating the expansion for $C^\bullet(x)$ in \eqref{eq:expansionCrooted}.

\medskip

Finally, the generating function of cubic planar graphs  $G(x) = \exp(C(x))$  has a singular expansion at $\rho$ of the form
$$
    G_0 + G_2X^2 + G_4X^4 + G_5X^5 + O(X^6),
$$
where $G_0  = e^{C_0}\approx 1.000604$ and $G_5 = e^{C_0} C_5 \approx -0.028837$.
An application of Theorem \ref{th:transfer} gives the estimate as claimed (analyticity in a $\Delta$-domain has been shown in the  proof of Theorem \ref{thm:count-connected}), where
$$
    g = 2G_5 /\Gamma(-3/2) \approx 0.061010.
$$
The probability that a random cubic planar graph is connected is then
$$
    p = c/g = e^{-C_0} \approx 0.999397.
$$
We remark that to get the right 6 decimal digits for $p$ we have actually computed $c$ and $g$ to higher precision.
 \qed

\subsection{Two-connected cubic planar graphs}\label{subsec:simple-2connected}

In order to enumerate  2-connected cubic graphs, we have to  discard the classes of networks that produce cut vertices, namely
$\mathcal{L}$ and $\mathcal{I}$.
We denote the corresponding classes with the same letters as in the previous two  sections, but one must be aware that they represent different classes, since they are restricted to 2-connected networks.
No confusion should arise as  we are not working with  connected and 2-connected networks at the same time.
The generating functions $S,P$ and $H$ have the same meaning as before, except that they are now restricted to 2-connected networks.
We have
\begin{equation}
    \def\arraystretch{2}
    \begin{array}{lll}
        D &=& S + P + H,\\
        S &=& D(D-S), \\
        P &=& x^2 D + x^2\ds\frac{D^2}{2}, \\
        H &=& \ds\frac{M(x,1+D)}{1+D}.
    \end{array}
\end{equation}

\paragraph{Proof of Theorem \ref{thm:count-2connected}.}
Using again  the relation $M(x,y) = (T(x^2y^3)-x^2y^3))/2$,
we get a single polynomial equation for $D$, which after a simple manipulation becomes
\begin{equation*}
H=F(x,D) = D  + \frac{x^2}{2}(1+D) - \frac{1}{2} T(x^2(1+D)^3) =0.
\end{equation*}
The equations
$$
	x^2(1+D)^3 = \tau,  \qquad F(x,D)=0
$$
have a unique positive  solution $\rho_b \approx 0.319523$ with $D_0 = D( \rho_b) \approx 0.010896$.
Eliminating $D$ from the previous equations gives the
minimal polynomial of $D$, which is the one in the statement.
As in the proof of Theorem~\ref{thm:count-connected} we check that $\rho$ is a root of $p(x)$, together with the remaining conditions on Lemma \ref{lem:technical}.
The proofs are along the same lines and are omitted to avoid repetition.
The Puiseux expansion at $\rho_b$ is
$$
	D(x) = D_0 + D_2 X^2 + D_3 X^3 + O(X^4), \qquad
X=\sqrt{1-x/\rho_b},
$$
with $D_3 \approx 0.233893$.

Let now $\overrightarrow{B}(x)$ be the generating function of 2-connected cubic planar graphs rooted at a directed edge. Then
$$
\overrightarrow{B}(x) = D(x) - x^2 D(x).
$$
The reason is that from the networks encoded by $D(x)$ we have to exclude the parallel networks with a double edge, that correspond to $x^2D(x)$. If now $B^\bullet$ is the generating function for 2-connected vertex-rooted cubic graphs, by double counting we have
$$
	B^\bullet(x) = \frac{\overrightarrow{B}(x)}{3}.
$$

Applying the Transfer Theorem we obtain for even $n$
$$
n\cdot b_n = n! [x^n] B^\bullet(x) \sim \frac{2(1 - \rho_b^2)D_3}{3\cdot \Gamma(-3/2)}\cdot n^{-5/2}\cdot \rho_b^{-n} n!,
$$
and from here the estimate on $b_n$  follows with $b = \frac{2(1 - \rho_b^2)D_3}{3\cdot \Gamma(-3/2)} \approx 0.059244.$
 \qed

\subsection{Cubic planar multigraphs}

Similarly to the simple case, we decompose connected cubic planar multigraphs using networks. In this situation we do not demand that removing the edge between the poles gives a simple graph. 
We also use the same notation for networks as before.
The equations are as follows.
\begin{equation}\label{eq:networks_multi}
    \def\arraystretch{1.6}
    \begin{array}{lll}
        D &=& L + S + P + H, \\
        L &=& \ds x^2 + x^2L + \frac{x^2}{2} (I + D - L), \\
        I &=& \ds\frac{L^2}{x^2}, \\
        S &=& D(D - S),\\
        P &=& x^2 + x^2D + x^2\ds\frac{D^2}{2},\\
        H &=& \ds\frac{M(x, 1 + D)}{1+D}.
    \end{array}
\end{equation}
The only differences with the system of equations describing the networks associated with simple graphs are the term $x^2$, in the equation for $P$, encoding the {\em 3-bond}, and the term $x^2(1+L)$, in the equation for $L$, encoding the cubic multigraph with two vertices and two loops, rooted at a loop and where the non-rooted loop is possibly replaced by a loop-network (see Figure \ref{fig:multi}).

\begin{figure}[htb]
\centering
    \raisebox{1.5ex}{\includegraphics[scale=1]{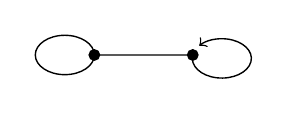}}
    \qquad
    \includegraphics[scale=1]{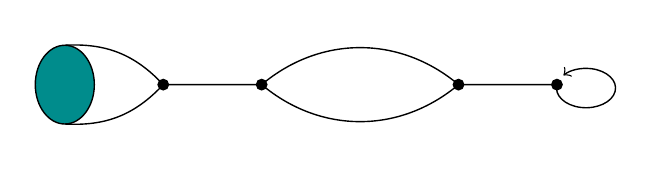}
\caption{Left is the only cubic multigraph with two vertices and two loops. Right is the same multigraph whose non-rooted loop has been replaced by a loop-network.}
\label{fig:multi}
\end{figure}
Using the same arguments as before, one can show that there exists a unique solution with non-negative coefficients of the above system, which is the combinatorial solution.

Let $C(x)$ be the generating function of  connected cubic planar multigraphs.
Due to the presence of  multiple edges and loops, there is no direct algebraic relation expressing $C^{\bullet}(x)$ in terms of networks. As in Section \ref{subsec:simple-general}, we need to resort once more to the Dissymmetry Theorem.

\paragraph{Proof of Theorem \ref{thm:count-multi}.}

We start by obtaining a single equation for $D$ from the system \eqref{eq:networks_multi}. First, we combine the second and the third equations and solve for $L$ as
$$
    L = 1 - \frac{x^2}{2} - \sqrt{\frac{x^4}{4} + 1 - x^2(D + 3)}.
$$
 Then we have
$$
    D = \frac{D^2}{1+D} + x^2 + x^2D + \frac{x^2}{2}D^2 + 1 - \frac{x^2}{2} - \sqrt{\frac{x^4}{4} + 1 - x^2(D + 3)} + \frac{M(x,1+D)}{1+D}.
$$
A simple manipulation together with \eqref{eq:TM} gives
\begin{equation}\label{eq:D_multi}
F(x,D)=    (1+D)\sqrt{\frac{x^4}{4} + 1 - x^2(D + 3)} - \frac{T\left(x^2(1 + D)^3\right)}{2} - 1 = 0.
\end{equation}
We rewrite as 
$$H\left(x,D,T\left(x^2(1+D)^3\right)\right)=\left(1+\frac{1}{2}T\left(x^2(1+D)^3\right)\right)^2-(1+D)^2\left(\frac{x^4}{4} + 1 - x^2(D + 3)\right)=0,
$$
where now $H$ is a polynomial.
We proceed as in the  proofs  of Theorems \ref{thm:count-connected} and \ref{thm:count-2connected}.
Equations \eqref{eq:D_multi} and $x^2(1+D)^3=\tau$  have a unique  positive solution
$$
    \rho_m \approx 0.250907 \text{ and } D_0 = D( \rho_m) \approx 0.187679.
$$
The minimal polynomial $p(x)$ of $D(x)$ is obtained by elimination and  is equal to the one in the statement.
We check that $\rho_m$ is a root of $p(x)$, together with the remaining analytic conditions of Lemma \ref{lem:technical}.

The rest of the proof is a further application of the Dissymmetry Theorem and is
very similar to that of Theorem \ref{thm:count-all} with some small changes.
The rooted tree-decompositions are the same, except that we have to update the corresponding classes to encode the {\em 3-bond} and the multigraph with two vertices and two loops.
Those changes only affect $\mathcal{M}$-nodes and $\mathcal{L}$-nodes.
The new equation for the generating function associated to $\mathcal{M}$-nodes is then
$$
    C_{\cal M} = \frac{x^2}{2}\left( 1 + D + \frac{D^2}{2} + \frac{D^3}{6}\right),
$$
As for $\mathcal{L}$-nodes, we need to introduce two new types of cut-vertices, those adjacent to a loop or to a double edge (see Figure \ref{fig:L-node-multi}).
The equation for the associated generating function becomes
$$
    C_{\cal L} = L + L^2 + \frac{L(D-L)}{2} + \frac{L^3}{6x^2}.
$$
Now when the tree is either rooted at an edge or at an oriented edge, we need to consider the new case when two cut-vertices are connected by a double edge (see the multigraph on the right of Figure \ref{fig:L-node-multi}).
The corresponding equations are given by
$$
    C_{{\cal L}-{\cal L}} = \frac{L^2}{2x^2} + \frac{L^2}{2}, \qquad
    C_{{\cal L}\rightarrow {\cal L}} = \frac{L^2}{x^2} + L^2.
$$
\begin{figure}[htb]
\centering
    \includegraphics[scale=1]{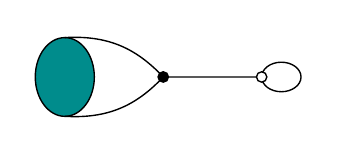}
    \qquad
    \includegraphics[scale=1]{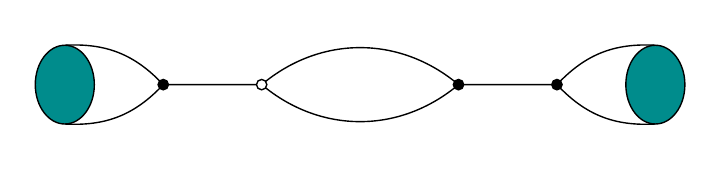}
    \caption{In white are the two new types of cut-vertices in a multigraph. That are respectively adjacent to a loop (left) and to a double edge (right).}
    \label{fig:L-node-multi}
\end{figure}

We then apply Theorem \ref{thm:species} and, after a straightforward calculation, obtain
\begin{equation}\label{eq:C-multi}
    \begin{split}
        C(x)
        & = \frac{x^2}{2}\left( 1 + D + \frac{D^2}{2} + \frac{D^3}{6}\right) + \overline{M}(x,1+D) + \frac{L^3}{6x^2}\\
        & - \frac{1}{2}\left( \log(1 - D + S) + D-S + \frac{(D-S)^2}{2} + P(S + H) + HS + \frac{P^2 + H^2}{2} + \frac{L^2}{x^2} + L^2\right).
    \end{split}
\end{equation}
The Puiseux expansion at $\rho_m$ is computed from that of $D(x)$ using the previous expression for $C(x)$ and is
of the form
$$
   C(x) =  C_0 + C_2X^2 + C_4X^4 + C_5X^5 + O(X^6).
$$
Since we do not have a singular expansion for $C^\bullet(x)$ that we can integrate as in the proof of Theorem \ref{thm:count-all}, we need to show directly  that $C_3=0$. Assume for contradiction that $C_3 \ne 0$.
Then, by  the Transfer Theorem, the ratio between the number of connected cubic planar multigraphs with $n$ vertices and the number of connected networks with $n$ vertices would tend to a  constant as $n$ goes to infinity. Let us  define a {\em bad edge} as either a double edge or a loop. For $n\geq 4$, a vertex of a connected cubic planar multigraph can be adjacent to at most one bad edge. Hence each  vertex is adjacent to at least one simple edge, hence there are at
least $n/2$ simple edges. Each time a simple edge of a connected cubic planar multigraph is distinguished and directed, we get a different connected network, hence the number of connected networks with $n$ vertices is at least $n/2$ times greater than the number of connected cubic planar multigraphs with $n$ vertices, which is a contradiction.

We proceed as in the last part of the proof of  Theorem \ref{thm:count-all}.
We compute
$$
    C_0 = C( \rho_m) \approx 0.070660 \text{ and } C_5 \approx -0.098979,
$$
together with the singular expansion of $G(x) = \exp( C(x))$, which is given by
$$
    G_0 + G_2X^2 + G_4X^4 + G_5X^5 + O(X^6),
$$
where
$$
    G_0 \approx 1.073217 \text{ and } G_5 \approx -0.106226.
$$
Finally, an application of Lemma \ref{th:transfer} gives the estimates as claimed.
As a corollary, the probability that a random cubic planar multigraph is connected is
$    p_m = h'/h\approx 0.931778$.
\qed

\paragraph{Remark.}  We provide here a short explanation for the similarity between equations \eqref{eq:rho} and \eqref{eq:rho-multi}. Let $p_1(x^2)$ be the polynomial in \eqref{eq:rho} and $p_2(x^2)$ that in \eqref{eq:rho-multi}. 
After making the change of variables $y=x^2$, $p_1(y)$ and $p_2(y)$ are obtained, by eliminating, respectively,  in the systems of equations
$$
\begin{array}{ll}
H_1=\textstyle\left(1+\frac{1}{2}T\left(y(1+D)^3\right)\right)^2-(1+D)^2\left(\frac{y^2}{4}+1-y(D-1)\right)=0, \qquad y(1+D)^3=\tau, \\
H_2=\left(1+\textstyle\frac{1}{2}T\left(y(1+D)^3\right)\right)^2-(1+D)^2\left(\frac{y^2}{4}+1-y(D+3)\right)=0, \qquad y(1+D)^3=\tau.
\end{array}$$
\medskip
Now rewrite $H_1$ and $H_2$ as 
$$
\begin{array}{ll}
H_1 =& \left(1+\frac{1}{2}T\left(y(1+D)^3\right)\right)^2-(1+D)^2\left(\frac{y^2}{4}+1\right)
-2y(1+D)^2-\tau, \\
H_2 =& \left(1+\frac{1}{2}T\left(y(1+D)^3\right)\right)^2-(1+D)^2\left(\frac{y^2}{4}+1\right)
+2y(1+D)^2-\tau.
\end{array}
$$
We deduce from here that $p_2(y)=p_1(-y)$, which is equivalent to the relation between \eqref{eq:rho} and \eqref{eq:rho-multi}.

\subsection{$P$-recursive sequences}

A series is $D$-finite if it satisfies a linear differential equation with polynomial coefficients. It is well-known (see Chapter 6 in  \cite{stanley2}) that $\{f_n\}$ is $P$-recursive if and only if
$\sum f_nx^n/n!$ is $D$-finite.

\paragraph{Proof of Theorem \ref{th:Dfinite}.} We show that in each case  the corresponding generating functions are $D$-finite.

\emph{Connected and 2-connected graphs}. The generating function $C'(x)$ is algebraic, hence it is $D$-finite \cite{stanley2}. It follows that $C(x)$ is also $D$-finite.
The same argument applies to the generating function $B(x)$ of 2-connected graphs.

\emph{Arbitrary graphs.} We use the same argument as in \cite{4-regular}, namely that if $C'(x)$ is algebraic then $\exp(C(x))$ is $D$-finite.
For completeness we briefly recall the proof. Let $G(x) = e^{C(x)}$.
One shows by induction  that $G^{(i)} = R_i(C',x) G(x)$, where
$R_i$ is a rational function in $C'$ and $x$.
Since $C'$ is algebraic, $\mathbb{Q}(C', x)$ is finite dimensional over $\mathbb{Q}(x)$, say of dimension~$k$.
Hence there are rational functions $S_i(x)$ such that
$
\sum_{i=0}^k  S_i(x) R_i(C',x) =0.
$
It follows that
$$
S_0(x)G + S_1(x) G' + \cdots + S_k(x) G^{(k)} = 0.
$$
proving  that $G$ is $D$-finite.

\emph{Multigraphs.}
In this case we cannot apply the previous argument, since there is no direct relation between the generating functions $D(x)$ and $C'(x)$. It follows from Equation \eqref{eq:D_multi} that $D(x)$ is algebraic. We  use Equation \eqref{eq:C-multi} to express $G(x) = \exp (C(x))$ in terms of $D(x)$, and the fact that the exponential of an algebraic function is $D$-finite, to obtain
$$
	G(x) = e^{C(x)} = J(x)  e^{ \overline{M}(x,1+D)},
$$
where $J(x)$ is a $D$-finite function (notice that the logarithm in \eqref{eq:C-multi} cancels with the exponential).
We next use the explicit expression \eqref{eq:MM} and the fact that $U(z)$ is algebraic to conclude that $\exp(\overline{M}(x,1+D))$ is $D$-finite (again a logarithm cancels). Since the product of $D$-finite functions is $D$-finite, we conclude that $G(x)$ is $D$-finite.
\qed

\section{Proofs of  limit law results: triangles}

In this section we obtain generating functions encoding triangles in cubic planar graphs and its distribution in  random cubic planar graphs. The main idea behind these proofs is that we are able to enrich the network decomposition of graphs in order to encode the number of triangles. More precisely, in order to study the distribution of the number of  triangles, we start with 3-connected cubic planar graphs.
By duality this amounts to studying vertices of degree 3 in triangulations.
The latter problem is solved in Section \ref{subsub: triang-triangul}.
We then use it to count triangles in networks in Section \ref{subsub:triang-net}.
In Section \ref{subsec:proofTriang} we perform the singularity analysis of the equations obtained in Section \ref{subsub:triang-net}, and complete the proof of Theorem \ref{thm:triangles}.
Finally, as a byproduct of the previous ideas, in Section \ref{subsec:triangle-free} we apply these tools  to enumerate planar cubic triangle-free graphs.
This does not follow directly from Theorem \ref{thm:triangles} as one needs to adapt the equations satisfied by the associated generating functions and perform a delicate analysis of singularities.

\subsection{Vertices of degree 3 in triangulations}\label{subsub: triang-triangul}
In this section we obtain the generating function of triangulations encoding the number of vertices of degree $3$.
This will be done by enriching the classical decomposition by Tutte  of triangulations in terms of 4-connected triangulations \cite{tutte}.

Throughout this section $\mathcal{T}^*$ denotes the class of triangulations not reduced to a triangle.
The associated generating function is $T^*(z) = T(z)-z$, where $T(z)$ is as in Equation \eqref{eq:Tu}.
Additionally $z^{-1}T^{*}(z)$ counts triangulations (not reduced to a triangle) in terms of \emph{internal} triangles.
Recall that $T_4(z)$ is the generating function of 4-connected triangulations, given in \eqref{eq:Tv}.
In both cases, $z$ encodes the number of vertices minus two.
A triangulation $A\in\mathcal{T}^*$ has a 4-connected core $C$, obtained by removing the vertices inside maximal separating triangles; the core is either a 4-connected triangulation or is isomorphic to $K_4$.
Then $A$ is obtained by possibly replacing the internal faces of $C$ with arbitrary triangulations.
This leads to the following equation, linking $T^*(z)$ and $T_4(z)$:
%
\begin{equation}\label{eqT}
    T^*(z) = \frac{T_4\left( z\left( 1 + z^{-1}T^*(z) \right)^2\right)}{1+ z^{-1}T^*(z)} + z^2 (1 + z^{-1}T^*(z))^3.
\end{equation}
The first term in the right hand-side   is equivalent to Equation (2.6) from \cite{tutte};
the second one corresponds to the case when the core is $K_4$.
Note that here we want to replace with triangulations in $\mathcal{T}^*$ instead of $\mathcal{T}$, as replacing a face with a single triangle amounts to doing nothing.  This is already encoded by the term 1 in $1 + z^{-1}T^*(z)$.

Our goal is to refine \eqref{eqT} by counting vertices of degree 3.
An \emph{internal} vertex in a triangulation is a vertex not incident with the root face, otherwise it is called \emph{external}.
Let $t(z,u)$ be the generating function of triangulations, where $z$ is as before and $u$ encodes internal vertices of degree 3.
In particular, $T^*(z) = t(z,1)$.
Let now $\mathcal{T}_0$ be the set of triangulations (except $K_4$) in which  the degree of the root vertex is equal to $3$, and  $\mathcal{T}_1$  those where the degree  is greater than 3.
Then we have $\mathcal{T}^* = \mathcal{T}_0 \cup \mathcal{T}_1 \cup \{K_4\}$.
Let $T_0(z,u)$ and $T_1(z,u)$ be the associated generating functions, where $u$ now counts the \emph{total} number  of vertices of degree $3$, including the external ones.
Then we have
\begin{equation*}
    T^*(z,u) = T_0(z,u) + T_1(z,u) + z^2u^4.
\end{equation*}

In the next lemma we obtain expressions for both $T_0(z,u)$ and $T_1(z,u)$:

%
%

%
%
\begin{lemma}\label{lem:triangular_maps}
The generating function $t = t(z,u)$ is defined implicitly in terms of $T_4(z)$ as
\begin{equation}\label{eq:t}
    t = \frac{T_4\Big( z\big( 1 + z^{-1}t \big)^2\Big)}{1+ z^{-1}t} + z^2\left( (1 + z^{-1}t)^3 + u-1 \right).
\end{equation}
In addition, we have
\begin{align}
    T_1(z,u)\,\, &=\,\, zu t, \label{eq:T1t3}      \\
        T_0(z,u)\,\, &=\,\,  (1+ 2zu - 3z)t  - z^2u. \label{eq:T0t3}
\end{align}
\end{lemma}
%
%
%
\begin{proof}
The first equation follows directly from \eqref{eqT}.
The only difference comes from the second term associated to $K_4$: when none of the internal faces is replaced with a triangulation, the central vertex has degree 3 and the configuration is encoded as~$u$.

When removing the root vertex (and the three adjacent edges) of a triangulation in $\mathcal{T}_1$, we obtain a smaller triangulation.
The reverse operation is to take a triangulation, draw a vertex on its root face, join it with the three vertices on the external face, and re-root the resulting map.
This gives \eqref{eq:T1t3}.

In order to obtain $T_0$ we first compute $T(z,u)$.
The following equation follows from \eqref{eq:t} by analyzing again the case where the core is $K_4$, and taking into account how many internal faces are replaced with triangulations:
\begin{equation*}
    T^*(z,u) = \frac{T_4 \left( z\big (1 + z^{-1}t \big )^2 \right )}{1+ z^{-1}t} + z^2\left( (1 + z^{-1}t)^3 -1 -3z^{-1}t + 3uz^{-1}t+ u^4 \right).
\end{equation*}
Finally, we use $T^*(z,u) = T_0(z,u) + T_1(z,u) + z^2u^4$, and after a simple computation we get \eqref{eq:T0t3}.
\end{proof}

\subsection{Triangles in networks}\label{subsub:triang-net}

 We are now back to labeled graphs and exponential generating functions.
In this section the goal is to obtain equations for networks encoding also triangles.
Here, variable $x$ marks vertices and $u$ marks triangles.
$D_i(x,u)$ is the generating function of non-isthmus networks in which the root edge belongs to exactly $i \in\{0,1,2\}$ triangles.
(observe that in a cubic graph there is no other possibility).
The same convention applies to series, parallel and $h$-networks.
The special case when the 3-connected core of an $h$-network is $K_4$
is encoded in the generating functions $W_i$.
We let $E(x,u)$ be the generating function of networks where triangles incident to the root edge are not counted, that is,
\begin{equation*}
    E(x,u) = D_0 + u^{-1}D_1 + u^{-2}D_2.
\end{equation*}
The next two lemmas provide the expressions for the series $D_i$, $S_i$, $P_i$, $W_i$, $I$, $L$ and for $H_0,\,H_1$ ($H_0,\,H_1$ will be treated separately as they are technically more involved).

\begin{lemma}\label{lem:networks-triangl}
The following equations hold:
\begin{equation*}
\renewcommand\arraystretch{1.5}
\begin{array}{lll}
    D_0 &=& S_0+P_0+W_0+L+H_0, \\
    D_1 &=& S_1+P_1+W_1+H_1,\\
    D_2 &=& P_2+W_2, \\
    I &=& \ds\frac{L^2}{x^2}, \\
    L &=& \frac{1}{2} x^2\left( I + E - L\right) + \frac{1}{2} x^2 (u-1)\left(x^2(E - L) + ux^2L + L^2) \right), \\
    P_0 &=& x^2(E - L)+\frac{1}{2}x^2(E - L)^2 , \\
    P_1 &=&  ux^2L(E-L)+u^2 x^2 L, \\
    P_2 &=& \frac{1}{2} u^2x^2L^2, \\
    S_0 &=& E \left( E - (S_0+u^{-1}S_1) \right) - u^{-1} S_1,\\
    S_1 &=& uL^3 + 2u x^2  L(E - L) + 2u^2 x^2L^2, \\
    W_0 &=& \frac{1}{2}x^4\left( 2(1+u)E^2 + 8E^3 + 5E^4 + E^5 \right), \\
    W_1 &=& \frac{1}{2} x^4\left(4u^2E + 6uE^2 + 2uE^3\right),\\
    W_2 &=& \frac{1}{2} x^4\left( u^4 + u^2E\right).
\end{array}
\end{equation*}
\end{lemma}

\begin{proof}
Equations for $D_0,\, D_1$ and $D_2$ are clear, since $S_2=P_2=H_2=0$.
The equation for $I$ is the same as in the univariate case.
The equation for $L$ is obtained as follows: from the main term $\frac{x^2}{2}(I+E-L)$ we need to consider separately three situations in which a new triangle is created: they are illustrated  in Figure~\ref{fig:loop-triangle}.
The corresponding generating functions are $\frac{1}{2} ux^4 (E-L)$, $\frac{1}{2} u^2x^4L$ and $\frac{1}{2} ux^2L^2$, hence the term $\frac{x^2}{2}(u-1)\left(x^2 (E-L)+\frac{x^2}{2}uL+\frac{1}{2}L^2\right).$

\begin{figure}[htb]
    \centering
    \includegraphics[scale=1.5]{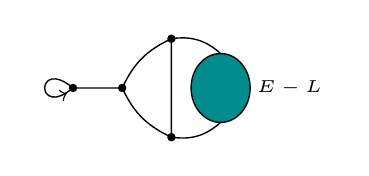} \hspace{0.05cm}
    \includegraphics[scale=1.5]{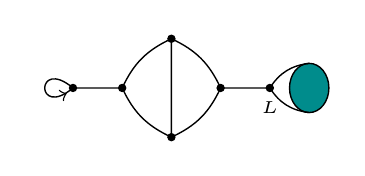} \hspace{0.05cm}
    \raisebox{-2ex}{\includegraphics[scale=1.5]{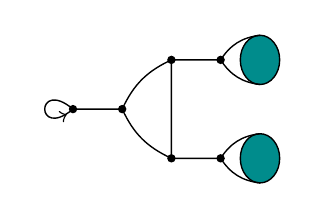}}
    \caption{The three configurations in $\mathcal{L}$ where an extra triangle is created. The associated generating functions are respectively $\frac{1}{2} x^4 u (E-L)$, $\frac{1}{2} x^4u^2 L$ and $\frac{1}{2} x^2 u L^2$.}
    \label{fig:loop-triangle}
\end{figure}

In the case of  parallel networks,  when using networks in $\mathcal{L}$  we create triangles incident with the root edge of the network.
The possible cases in $P_1$ and $P_2$ are illustrated  in Figure~\ref{fig:parallel}.

\begin{figure}[htb]
    \centering
    \includegraphics[scale=1.3]{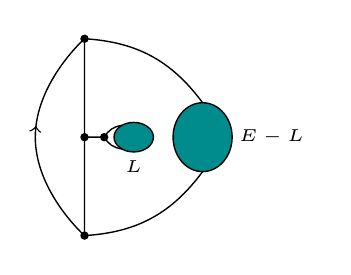} \hspace{0.05cm}     \includegraphics[scale=1.3]{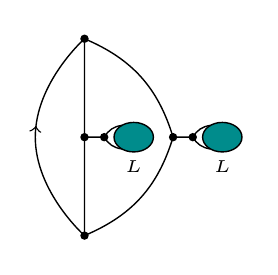}
    \caption{Contributions of $ux^2 L(E-L)$ to $P_1$ and $\frac{1}{2}x^2u^2L^2$ to $P_2$.}
    \label{fig:parallel}
\end{figure}

The equation for $S_1$ follows by considering the possible cases in which the root edge is incident with a triangles, as described  in Figure~\ref{fig:S1-triangle}.

\begin{figure}[htb]
    \centering
    \includegraphics[scale=1.3]{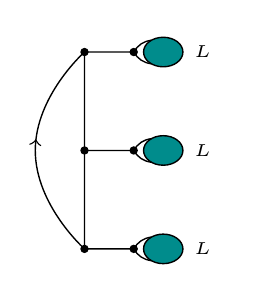} \hspace{0.05cm}     \includegraphics[scale=1.3]{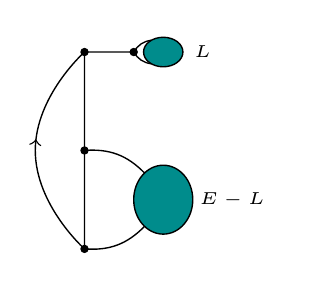} \hspace{0.05cm}      \includegraphics[scale=1.3]{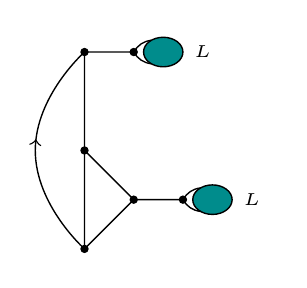}
    \caption{Contributions to $S_1$: the corresponding generating functions are $uL^3$,  $u x^2  L(E - L)$,  $u^2 x^2L^2$. For the second and third configuration there are two possibilities.}
    \label{fig:S1-triangle}
\end{figure}

The equation for $S_0$ is obtained as in the univariate case, by subtracting the term $u^{-1}S_1$.
Finally, the equations for $W_0, W_1$ and $W_2$ are obtained by considering all cases where $K_4$ is the core of the $h$-network. Observe that the different coefficients that appear in the expressions of $W_0$, $W_1$ and $W_2$ are due to symmetries of $K_4$.
\end{proof}

The previous system can be easily rewritten as we have done earlier (see the paragraph after the statement of Lemma 13) so that the right-hand terms have non-negative coefficients, and thus admits a unique non-negative power series as solution.
The following lemma gives the expression for $H_0$ and $H_1$ in terms of $T_0(z)$ and $T_1(z)$.
Joint with the previous lemma, this completes the system of equations encoding triangles:
\begin{lemma}\label{lem:eq-H0H1-triangl}
Let $t(x,u)$ be the generating function defined by Equation~\eqref{eq:t}.
Then $H_0$ and $H_1$ are given by the following expressions:
$$
    \renewcommand\arraystretch{2.5}
    \begin{array}{lll}
        H_1(x,u) &=& \ds\frac{1}{2} x^2 u\cdot t\left( x^2(1+E)^3, 1+\frac{u-1}{(1+E)^3} \right),
        \label{eq:H1} \\
        H_0(x,u) &=& \ds\frac{1}{2} t\left(x^2(1+E)^3,1+\frac{u-1}{(1+E)^3}\right)\frac{1-x^2(E-2u+3)}{1+E}- \frac{1}{2}x^4(1+E)^2((1+E)^3+u-1)).
        \label{eq:H0}
    \end{array}
$$
\end{lemma}

\begin{proof}
We say that a triangle in a network  is \emph{external} if it is incident with the root edge.
The edges of an external triangle that are not the root edge are called \emph{special}.

We denote by $\mathcal{M}_0$ and $\mathcal{M}_1$ the family of edge-rooted 3-connected cubic planar graphs (except $K_4)$ without external triangles and with one external triangle, respectively,
and let $M_0(x,y,u)$, $M_1(x,y,u)$ be the associated generating functions, where $x$, $y$ and $u$ mark vertices, edges and triangles, respectively.
Similarly to Equation~\eqref{eq:TM} we have that
$$
    M_0(x,y,u) = \frac{1}{2}T_0(x^2y^3,u),\qquad M_1(x,y,u)=\frac{1}{2}T_1(x^2y^3,u).
$$
Let $m_1(x,y,u)=M_1(x,y,u)/(uy^3)$, where now $u$ counts non-external triangles, and $y$ counts the number of edges minus three (we do not count the root edge and the special edges).

A network in $\mathcal{H}_1$
is obtained from a graph  $G$ in $\mathcal{M}_1$ in which we replace edges (except the root edge) with networks, and where the three edges of the external triangle of $G$ are not replaced (recall that the external triangle is the only triangle incident with the root edge).
Observe that triangles in $G$ are isolated (because $G$ is $3$-connected), hence there are no triangles sharing edges and the previous replacement can be made.
%
%
%
In particular, the term $u+3E+3E^2+E^3=(1+E)^3+u-1$ encodes the substitution of networks on 3-sets of edges defining triangles (except the external triangle and the corresponding edges, which are not substituted). 
This translates into the equation
$$
    H_1(x,u)=u \cdot m_1\left(x,1+E, 1+\frac{u-1}{(1+E)^3}\right).
$$
The expression for $H_1$
 is obtained by writing first $m_1$ in terms of $T_1$, and then  writing $T_1$ in terms of $t$ using~\eqref{eq:T1t3}.

Let us now consider a network in $\mathcal{H}_0$.
It can be obtained in two different ways: either from  a core without an external triangle, or from a core with an external triangle in which  some special edges are replaced with a non-empty network.
Using a similar encoding argument as before we arrive at
\begin{equation*}
    H_0(x,u) = \frac{M_0\left(x,1+E,1+\ds\frac{u-1}{(1+E)^3}\right)}{1+E}+ (2E+E^2)\cdot m_1\left(x,1+E,1+\frac{u-1}{(1+E)^3}\right),
\end{equation*}
where the factor  $2E+E^2$ in the second summand corresponds to the substitution of networks on the pair of special edges.
Using the expressions of $M_0$ and $m_1$ in terms of $T_0$ and $T_1$, and  Equations \eqref{eq:T1t3} and \eqref{eq:T0t3}, after simplification  we get the expression for $H_0$, as claimed.
\end{proof}

We conclude this section by expressing the generating function of vertex-rooted graphs  $C^{\bullet} (x,u)$, where $x$ marks vertices and $u$ marks triangles, in terms of networks:
\begin{equation}\label{eq:rooted-conn-triang}
	3C^{\bullet} (x,u)=D_0+D_1+D_2 + I - L-  x^2 (D_0+D_1+D_2) - L^2.
\end{equation}
This equation is obtained by considering all networks (which is counted by $D_0+D_1+D_2 + I$) and removing those where the root edge is either a loop or a multiple edge (term $L + x^2 (D_0+D_1+D_2)+ L^2$ ). This difference is equal to the generating function for networks with only simple edges, which by double counting it is equal to $3C^{\bullet}(x,u)$.


\subsection{Singularity analysis and proof of the main result}\label{subsec:proofTriang}

After obtaining the system of equations in Lemmas \ref{lem:networks-triangl}, \ref{lem:eq-H0H1-triangl} and Equation~\eqref{eq:rooted-conn-triang} we proceed to analyze it.
In order to apply Lemma \ref{th:quasi-powers} for proving  a Gaussian limit law, our first task is to find the dominant singularities (of $x$ as a function of $u$, for $u$ close to $1$) of the function $C(x,u)$  counting triangles in connected cubic planar graphs.
We start finding the singularities for 3-connected graphs. Later, we use the results from the previous section to obtain the singularities of $C(x,u)$.

\paragraph{Singularities of 3-connected graphs.}
Recall that the generating function $t(z,u)$ encodes triangulations, where $z$ is the number of vertices minus 2 $u$ encodes internal vertices of degree 3. Its expression in terms of $T_4(z)$ is given in Lemma \ref{lem:triangular_maps}.
The next result gives the dominant singularities of
the generating function of 3-connected graphs:

\begin{lemma} \label{lem:singularities-t}
	Let $t(z,u)$ be as in Lemma \ref{lem:triangular_maps}.
	Let $u$ be a fixed complex number with $|u-1|<\varepsilon$, where $\varepsilon>0$ is sufficiently small.
	Then the point $z_0 = z_0(u)$ where $t(z,u)$ ceases to be analytic is the solution of the following equation:
\begin{equation}\label{eq:triangle_sing_curve}
	z_0(1 +(u - 1)z_0)^2 = \frac{27}{256}.
\end{equation}
Moreover, at the critical point $(z_0,u)$ we have the relation:
\begin{equation}\label{eq:triangle_t3}
    (z_0(u - 1) + 1)t(z_0,u) = \frac{1}{8} - z_0(1 + (u - 1)z_0).
\end{equation}
\end{lemma}

\begin{proof}
The unique singularity of $T_4(z)$ is  at $4/27$ (see Section 2).
Hence, for $u$ in a small neighborhood of $1$, the only possible source of singularities for $t(z,u)$ in Equation \eqref{eq:t} comes from the singularity of $T_4(z)$, giving the relation
\begin{equation*}
	z_0(1 + z_0^{-1}t(z_0,u))^2 = 4/27.
\end{equation*}
We also know that $T_4(4/27) = 7/5832$, hence at the singular point we have
\begin{equation*}
	t(z_0,u) = \frac{7/5832}{1 + z_0^{-1}t(z_0,u)} + z_0^2((1 + z_0^{-1}t(z_0,u))^3 + u - 1).
\end{equation*}
Eliminating $t(z_0,u)$ from the previous two equations gives \eqref{eq:triangle_sing_curve}, and an elementary computation gives Equation  \eqref{eq:triangle_t3}.
\end{proof}

\paragraph{Singuarities of connected graphs.}

We have seen in the proof of Theorem \ref{thm:count-connected} that the singularities of the generating function $D(x)$ of cubic networks come from the singularities of $T(z)$.
Variable $u$ marks triangles, which is a linear parameter. Hence, by continuity and for $u$ sufficiently close to 1, this also holds for the bivariate generating functions of networks.
For a given $u$ close to 1, we let $\rho(u)$ be the dominant singularity of the function $E(x,u)$.
Notice that, because of \eqref{eq:rooted-conn-triang}, it is also that of $C(x,u)$; there is no cancellation because there is none for $u=1$. 
Remark also that $\rho(1) $ is equal to the constant $\rho\approx 0.3192246062$ from Theorem \ref{thm:count-connected}.

In order to  determine $\rho(u)$, we  find two equations satisfied by $u$, $\rho(u)$ and $E(\rho(u),u)$.
Then eliminating $E$ will give us $\rho(u)$ implicitly in terms of $u$.
Once we have access to $\rho(u)$, an application of Lemma \ref{th:quasi-powers} will give the asymptotic normal law with the corresponding moments.

\begin{lemma}\label{lem:eqns-sing-triangles}
	For  fixed $u$ close to 1, $E(x,u)$ admits two dominant singularities given by the two curves $\pm\rho(u)$ and such that $\rho(1) = \rho$.
	As $x\to \rho(u)^-$, we have locally
	\begin{equation*}
		E(x,u) = E_0(u) + E_2(u)\left(1 - \frac{x}{\rho(u)}\right) + E_3(u)\left(1 - \frac{x}{\rho(u)}\right)^{3/2} + \ldots,
	\end{equation*}
	where $E_0(u)$, $E_2(u)$ and $E_3(u)$ are analytic functions.

    Let  $x = \rho(u)$ be the {positive} dominant singularity of $E(x,u)$
    and let $E = E(x,u)$.
    Then the following two equations hold:
    \begin{align}
         x^2(1+E)^3 \left(1+(u-1)x^2\right)^2\,\,\, &=\,\, \frac{27}{256}, \label{eq:first-rho} \\
         256(1+(u-1)x^2)^2(1+E)A\,\,\, &=\,\, 256x^2(1+(u-1)x^2)^3(E^3+3E^2+3E) + B, \label{eq:second-rho}
    \end{align}
	where
	\begin{equation*}
    	\begin{array}{lll}
        	A &=&\left( (u^2x^4 - 2ux^4 + x^4 - x^2 - 2)^2-4x^2(1+(u-1)x^2)^2E \right)^{1/2}, \\
         	B &=&  256(u-1)^3x^8 + 768(u-1)^2x^6 +192(u-1)(3u+1) x^4 +(1066u-810)x^2+517.
    	\end{array}
	\end{equation*}
\end{lemma}

\begin{proof}
	Similarly to the univariate case, we show that the only source of singularities for $E = E(x,u)$ comes from $t = t(x^2(1+E)^3, 1 + (u-1)/(1+E)^3)$. 
	Furthermore, the singular behaviour of $t$  transfers directly to that of $E$.
	In our case, both statements can be  deduced directly from a slightly modified version of \cite[Theorem 2.31]{drmota}, in which we now require that $|P_E(t(\tau), E(\rho), \rho, 1)| \neq 0$ when $u=1$, and that $t(z,u)$ admits a 3/2 singular behaviour locally around $u=1$ and $z = \tau(u)$, where the $\tau(u)$ is the solution of $z_0$ in \eqref{eq:triangle_sing_curve}.
	By elimination from the equations in Lemmas \ref{lem:networks-triangl} and \ref{lem:eq-H0H1-triangl}, we obtain a polynomial equation $P(t,E,x,u)=0$, which has degree 6 in $E$ (it is  too big to be displayed here).
	From there, we check that $|P_E(t(\tau), E(\rho), \rho, 1)|\approx 7.1818705965$.
	For the second condition, we eliminate $V(z)$ and $T_4(z)$ from \eqref{eq:Tv}, \eqref{eq:t} and $z = V(z)(1-V(z))^2$ to obtain an irreducible polynomial equation $Q(t(z,u),z,u)=0$. 
	Using Newton's polygon algorithm on $Q$ (as it is square-free), we  compute the Puiseux expansion of $t(z,u)$ locally around $z=\tau(u)$, which is of the form:
	\begin{equation*}
		t(z,u) = t_0(u) + t_2(u)\left(1 - \frac{z}{\tau(u)}\right) + t_3(u)\left(1 - \frac{z}{\tau(u)}\right)^{3/2} + \ldots,
	\end{equation*}
	where $t_0(u)$, $t_2(u)$ and $t_3(u)$ are analytic functions.

	Let us finally consider the expressions for $H_0$ and $H_1$ in Lemma \ref{lem:eq-H0H1-triangl}.
	Since the singularities of $E$ must come from the substitution in $t(z,u)$, the point $(z_1,u_1) = (x^2(1 + E)^3, 1 + (u-1)/(1 + E)^3)$ must be a singular point of $t(z,u)$.
	The singularities of $t(z,u)$ are given by the relation \eqref{eq:triangle_sing_curve}, hence we have:
	\begin{equation}
    	z_1\left(1 + (u_1 - 1)z_1\right)^2 = x^2(1 + E)^3 \left(1 + \left(1 + \frac{(u - 1)}{(1 + E)^3} - 1\right)x^2(1 + E)^3\right)^2 = \frac{27}{256},
	\end{equation}
	which is precisely \eqref{eq:first-rho}.
	Let us now deduce Equation \eqref{eq:second-rho}.
	We first need the evaluation of $t(z,u)$ at the point $(z_1,u_1)$.
	This follows directly from \eqref{eq:triangle_t3} and \eqref{eq:first-rho} and gives:
	\begin{equation*}
    	t(z_1,u_1) =  \frac{32(u - 1)x^2 + 5}{256(1 + (u - 1)x^2)^2}.
	\end{equation*}
	Notice that all the functions, involved in both Lemmas \ref{lem:networks-triangl} and \ref{lem:eq-H0H1-triangl}, can be written in terms of $E, L$ and the variables $x$ and $u$.
	Solving for $L$ and substituting  provides a second equation on $E$, $x$, and $u$.
	The solution for $L$ is given by:
	\begin{equation}\label{eq:L-triang}
    	L(x,u) = \frac{x^2 + 2   - (u - 1)^2x^4-A}{2(1+ (u - 1)x^2)},
	\end{equation}
	where $A$ is an in the statement.
	It remains finally to write $D_0$, $D_1$, $D_2$ in terms of $E$, $L$, $x$ and $u$, then to replace $L$ with the expression in \eqref{eq:L-triang}, and to perform an elementary computation to obtain \eqref{eq:second-rho}.
\end{proof}

\paragraph{Proof of Theorem \ref{thm:triangles}.}

One can eliminate $E$ from the system composed of Equations \eqref{eq:first-rho} and \eqref{eq:second-rho} to obtain a single polynomial equation $p(x,u)=0$ in $x$ and $u$, whose smallest positive solution in $x$ is the singularity $\rho(u)$ of $E(x,u)$.
The polynomial $p$ has degree 40 in $x^2$ and is too large to be displayed here.
We then  differentiate $p(\rho(u),u)$ with respect to $u$ and compute the following values (using \textsc{Maple}):
$$
    \rho'(1) = -0.0389371919, \qquad \rho''(1) = 0.0229417852.
$$
Alternatively, we can differentiate \eqref{eq:first-rho} and \eqref{eq:second-rho} and solve the corresponding system involving $\rho(1)$ and $\rho'(1)$, and similarly for $\rho''(1)$.

In order to apply  Lemma \ref{th:quasi-powers}, we need to show that $E(x,u)$ is analytic in a $\Delta$-domain at $x=\rho(u)$.
By Lemma \ref{lem:eqns-sing-triangles}, $E(x,u)$ has an expansion in powers of $\sqrt{1-x/\rho(u)}$ for $u$ near 1. 
It is hence analytic in a sufficiently small neighborhood of $\rho(u)$ sliced along the ray $[\rho(u),\infty]$.
Consider now $u$ in a small neighborhood $U$ of 1, and take $u_0 \in U$ real with $\rho(u_0)>|\rho(u)|$.
By the same argument as in the proof of the univariate case (Theorem \ref{thm:count-connected}), $E(x,u)$ is analytic in a $\Delta$-domain at $u_0$.
It follows that $E(x,u)$ is analytic in a $\Delta$-domain at $\rho(u)$.
Thus, for $u$ in a small neighborhood of $1$, we get the estimate
$$
    [x^n]E(x,u) = c(u)\cdot n^{-5/2} \rho(u)^{-n} \left(1 + O(n^{-1})\right).
$$
By a direct application of Lemma \ref{th:quasi-powers}, we are able to first compute the values $\rho'(1)$ and $\rho''(1)$, then the values of $\mu$ and $\lambda$, as claimed.
This concludes the proof of Theorem \ref{thm:triangles}.
\qed

\subsection{Enumeration of triangle-free cubic planar graphs}\label{subsec:triangle-free}

In this section we enumerate cubic planar graphs without triangles.
The starting point is the enumeration of triangles from the previous section (more precisely, Lemmas \ref{lem:networks-triangl} and \ref{lem:eq-H0H1-triangl}).
We consider cubic planar networks without triangles, except possibly the ones incident with the root edge: when removing the root edge of the network, the resulting graph becomes triangle-free. In particular, we need to encode such networks where the only triangles are incident with the root edge (and becoming triangle-free when replaced at an edge).

We use the same notation as in the previous section, with the difference that now we do not encode triangles.
The following lemma gives  the relations between the various classes of  networks in this setting.

\begin{lemma} \label{lem:networks-free}
Let $t(x,u)$ be the generating function defined by Equation~\eqref{eq:t}, and let $E$, $L$, $I$, $D_i$, $S_i$, $P_i$, $J_i$, $W_i$ be the generating functions of networks without triangles except those containing the root.
Then
\begin{equation}\label{eq:networks-free}
    \renewcommand\arraystretch{1.6}
    \begin{array}{lll}
        E &=& D_0 + D_1+D_2,\\
        D_0 &=& S_0+P_0+W_0+L+H_0, \\
        D_1 &=& S_1+P_1+W_1+H_1,\\
        D_2 &=& P_2+W_2, \\
        I &=& \frac{L^2}{x^2}, \\
        L &=& \frac{1}{2} x^2\left( I + E - L\right) - \ds\frac{1}{2} x^2 \left(x^2(E - L)  + L^2) \right), \\
        P_0 &=& x^2(E - L)+\frac{1}{2}x^2(E - L)^2,\\
        P_1 &=& x^2L(E - L),\\
        P_2 &=& \frac{1}{2} x^2L^2\\
        S_0 &=& E\cdot \Big( E - (S_0+S_1) \Big) - S_1,\\
        S_1 &=& L^3 + 2x^2(E - L)L \\
        W_0 &=& \frac{1}{2}x^4\left(2E^2 + 8E^3 + 5E^4 + E^5\right),\\
        W_1 &=& \frac{1}{2}x^4\left(6E^2 + 2E^3\right),\\
        W_2 &=& \frac{1}{2}x^4E,\\
        H_1 &=& \frac{1}{2}x^2 \cdot t\left( x^2(1+E)^3, 1-\frac{1}{(1+E)^3} \right),\\
        H_0 &=& \frac{1}{2}t\left(x^2(1+E)^3,1-\frac{1}{(1+E)^3}\right)\frac{1-x^2(E+3)}{1+E} - \frac{1}{2}x^4(1+E)^2((1+E)^3-1)).
    \end{array}
\end{equation}
\end{lemma}

\begin{proof}
The equations are obtained from the ones in Lemmas \ref{lem:networks-triangl} and \ref{lem:eq-H0H1-triangl} as follows.
For those defining generating functions with an index $i$ (counting networks whose root edge is incident with $i$ triangles) we take the corresponding equation, divide it by $u^i$ and then set $u=0$.
For the two respectively defining the functions $I$ and $L$, because they only count networks without any triangle, we simply set $u=0$.
In particular, with this convention, equation for $D_0$, $D_1$, $D_2$, $P_0$ and $I$ are exactly the same as in Lemma \ref{lem:networks-triangl}.
Equation for $L$ is obtained from Equation for $L$ in Lemma \ref{lem:networks-triangl} by simply writing $u=0$.

 The remaining  equations are  obtained by adapting the argument in Lemmas \ref{lem:networks-triangl} and \ref{lem:eq-H0H1-triangl} in this setting. For instance, the equation for $S_0$ follows since $S_1$ has been replaced with $u^{-1}S_1$ (namely, the generating function $u^{-1}S_1$ in Lemma \ref{lem:networks-triangl} in the equations we have now is written as $S_1$).
\end{proof}

\paragraph{Proof of Theorem \ref{thm:count-triangle-free}.}
The proof uses a simple variant of Lemma \ref{lem:technical} as follows. In the present situation we have an equation of the form
$$
	F(x,E(x)) = f(x,E(x)) + t\left(x^2(1+E)^3,1-\frac{1}{(1+E)^3}\right) = 0.
$$
The singularities of $t(z,u)$ are given by Equation \eqref{eq:triangle_sing_curve}, that is, $z(1+(1-u)z)^2= \tau$.
Assuming that the dominant singularities of $E(x)$  come from those of $t(z,u)$, they are obtained by solving
\begin{equation}\label{eq:first-rho-free}
 F(x,E) =0, \qquad    x^2(1+E)^3 \left(1-x^2\right)^2 = \tau.
\end{equation}
The conditions to be verified are the same as in Lemma \ref{lem:technical}, except that the equations determining the singularity are now \eqref{eq:first-rho-free} instead of $F = x^2(1+E)^3=0$.
They have as  smallest positive solution
$$
    \zeta\approx 0.378537, \qquad E_0=E(\zeta) \approx 0.000951.
$$
We verify the three technical conditions from Lemma \ref{lem:technical} and obtain the corresponding estimate on $[x^n]E(x)$ by computing the singular expansion:
\begin{align*}
	& E(x) = E_0 + E_1 X + E_2 X^2 + E_3 X^3 + O(X^4), &
	& \text{where } X = \sqrt{1-x/\zeta}, &
	& \text{and } E_3 \approx 0.094744.
\end{align*}

Let now $F(x)$ be the generating function of connected triangle-free cubic planar graphs.
We have the relation
$$
	3F^\bullet(x) = D_0 +I - L - L^2 - x^2(E-L),
$$
where $F^{\bullet}(x)$ counts connected triangle-free cubic planar graphs rooted at a vertex.
The reason is that the only networks that contribute to $F^\bullet(x)$ are those with no triangle incident with the root edge, that is, those counted by $D_0$ and $I$.
We have to subtract those which are not simple, namely: loop networks ($L$), series composition of two loop networks ($L^2$), and parallel composition of a double edge and a non-loop network ($x^2(E-L))$.
From the expansion of $E(x)$ at $\zeta$ we get a corresponding expansion of $F^\bullet(x)$ with $F^\bullet_3 \approx 0.001077$.
Since $E(x)$ is $\Delta$-analytic at $\pm\zeta$, so is  $E(x)$. Thus we obtain the following estimate for $n$ even:
$$
n f_n \sim n! [x^n] F(x) \sim \frac{2F^\bullet_3}{\Gamma(-3/2)}\cdot n^{-5/2}\zeta^{-n}  n!,
$$
which gives the estimate for $f_n$ claimed in Theorem \ref{thm:count-triangle-free}.

Finally, the exponential formula gives  $\sum_{n\ge0}t_n \frac{x^n}{n!} = e^{F(x)}$, from which can we compute the number of arbitrary triangle-free graphs up to any number of vertices.
We find an approximation to  the limiting probability $p$ of a triangle-free graph being connected by computing the quotients $f_n/t_n$, and obtain $p \approx 0.99995$. Finally we compute $\alpha = p \cdot f \approx  0.0009109$.
\qed

\section{Proofs of  limit law results: cherries and bricks}
The strategy in the proofs of Theorems \ref{thm:cherries} and \ref{thm:bricks} is very similar than the one of Theorem \ref{thm:triangles}.
We use the variable $u$ to mark the number of copies of a given cherry or brick, and we find the equations satisfied by the respective bivariate generating functions $D(x,u)$ of non-isthmus networks.
We then find the dominant singularities of $D(x,u)$ and apply Lemma \ref{th:quasi-powers} to deduce an asymptotic normal law, together with a computation of the first two moments.

\paragraph{Proof of Theorem~\ref{thm:cherries}.} 
Let  $H$ be a fixed cherry with $h$ vertices, and let $\aut(H)$ be the number of automorphisms of $H$.
Notice that the unique vertex of degree 1 in a cherry must be fixed by every automorphism.
We let variable $x$ mark vertices and $u$ mark copies of $H$ in a cubic network.

Clearly, copies of a cherry can only arise in loop networks. 
This implies that in order to obtain an equation for $D(x,u)$, we need only modify the corresponding equation for $L$ in Lemma \ref{lem:networks}, as follows
$$
    L = \frac{x^2}{2}(D+I-L)+\frac{x^h}{\aut(H)}(u-1).
$$
This is because occurrences of $H$ are encoded with the monomial $\frac{h!}{\aut(H)}\frac{x^h}{h!} = \frac{x^h}{\aut(H)}$, since $h!/\aut(H)$ is the number of ways of labeling $H$, and each occurrence is marked by $u$.

We solve for $L$ and, as in the  proof of Theorem \ref{thm:count-connected} of Section \ref{subsec:simple-connected}, we obtain an equation for $D = D(x,u)$, namely
\begin{equation}\label{eq:cherries}
	(1+D)\sqrt{\frac{x^4}{4} + 1 - x^2(D-1)-\frac{2x^h}{\aut(H)}(u-1)} - \frac{T(x^2(1+D)^3)}{2} - 1 = 0.
\end{equation}
By continuity, for $u$ close to 1, the dominant singularities of $D$ come from those of $T$ and we have
$$
	x^2(1+D)^3 = \tau =\frac{27}{256}.
$$
Since $T(\tau)=1/8$, we get a second equation from \eqref{eq:cherries} at the singular point $(x,u)$.
Eliminating $D$ we obtain a polynomial equation $P(x,u)=0$ for the dominant singularities $x=\rho(u)$ (since all functions are even in $x$, $-\rho(u)$ is also a singularity as in the univariate case).
Similarly to Theorem \ref{thm:triangles}, one shows that $D$ is analytic in a $\Delta$-domain at $\rho(u)$.
A routine computation gives the values of the moments as claimed.

\medskip

It remains to show that $\lambda >0$.
Taking the approximation $\rho \approx 0.32$ and setting $a=\aut(H)$, we obtain
$$
    a^2\lambda =  1.68 a\rho^h + (3.05-5.64h)\rho^{2h} \ge \rho^h(1.68 -5.64h\rho^{h}).
$$
It is elementary to check that the right-hand side is positive for $h \ge 2$.
\qed

\paragraph{Proof of Theorem~\ref{thm:bricks}.}
Let $B$ be a brick and $\aut(B)$ be the number of automorphisms of $B$ with the convention that the two vertices of degree 2 are distinguishable.
A brick $B \ne K_4^-$ can only arise from an $h$-network isomorphic to $B$ in which no edge is replaced.
The modified generating functions of networks becomes
$$
	H = \frac{M(x,1+D)}{1+D} + \frac{x^b}{\aut(B)}(u-1).
$$
Notice that when $B=K_4^-$, the brick $B$ can also appear as a series composition of two loop networks, and we have to modify $P$ in consequence:
$$
	P = x^2D + \frac{x^2}{2} D^2 + \frac{x^2}{2}L^2(u-1).
$$
The singularities are determined as before by modifying the equation satisfied by $D$.
Again, an application of Lemma \ref{th:quasi-powers} and a routine computation give the claimed result.

To finally show that $\lambda >0$, we approximate $\rho\approx 0.32$ as in the previous proof and obtain
$$
    b^2 \lambda = 1.74b \rho^b - (3.18 + 6.08h)\rho^{2b}.
$$
It is also elementary to check that the right-hand side is positive for $b \ge2$.
  \qed


\subsection*{Numerical table}

The numbers of arbitrary, connected and 2-connected cubic planar graphs for small values of $n$ were given in Table 1 from \cite{cubic}. To these we add the  new families of cubic planar graphs we have enumerated in this work: multigraphs and triangle-free.

$$
\begin{array}{|r|r|r|}
\hline
n & \hbox{Multigraphs}  & \hbox{Triangle-free} \\
\hline
 2&  							  2  & 0 \\
  4&                             47& 0 \\
  6&                            4710& 0 \\
  8&                          1239875& 840 \\
  10&                         669496590 & 181440\\
 12&                         634267800705 & 79833600\\
  14 &                      946240741175730 & 61232371200\\
 16 &                     2056603172557758825 & 69529227768000\\
  18 &                    6148723823146399745250 & 105801448533580800\\
  20 &                  24214871726535475276466175 & 205703216403561676800\\
   22 &               121481234613567346345100623350 & 497546215788236719104000\\
  24 &              756128791200007319214204305696475 & 1467149457794547540581568000\\
   26 &           5716103221856552423681553448136208750 & 5185994162876896958824435200000\\
    28 &       51574528549599744692080383726773969240625 & 21659841687523558647126605967360000\\
  30&       547355299046868963962856204715812138973841250 & 105576624440793627977398205974671360000\\
\hline
\end{array}
$$

\subsection*{Acknowledgments}
This research was started while the first author was visiting the Department of Mathematics at the Freie Universit\"at Berlin, and continued in a second visit. He is grateful for the support and hospitality received during his stays at the research group of Tibor Szab\'o at FU. We also thank Marc Mezzaroba for useful discussions on some computational aspects of our work. Finally, we thank the anonymous referees for providing many useful suggestions and comments.


\bibliography{biblio}
\bibliographystyle{abbrv}

\end{document}